\documentclass[leqno,12pt]{amsart}
\usepackage{amssymb,amsthm,amsmath,txfonts,a4wide}
\usepackage{xcolor}
\usepackage{enumerate}
\usepackage{hyperref}
\usepackage{tikz}
\usetikzlibrary{shapes,arrows}
\usepackage{times}
\usepackage{verbatim}

\hypersetup{
colorlinks,
citecolor=green,
filecolor=blue,
linkcolor=red,
urlcolor=blue
}

\newtheorem{defn}{Definition}[section]
\newtheorem{lemma}[defn]{Lemma}
\newtheorem{remark}[defn]{Remark}

\newtheorem{theorem}[defn]{Theorem}
\newtheorem{corollary}[defn]{Corollary}

\newtheorem{proposition}[defn]{Proposition}

\newtheorem{claim}[defn]{Claim}

\numberwithin{equation}{section}

\begin{document}
	
\title{Criteria for finite time blow up for a system of Klein-Gordon equations}
\author{Yan Cui}
\address{Department of Mathematics, Jinan University, Guangzhou, P. R. China.}
\email{cuiy32@jnu.edu.cn}

\author{Bo Xia}
\address{School of Mathematical Sciences,USTC, Hefei, P. R. China.}
\email{xiabomath@ustc.edu.cn,xaboustc@hotmail.com}

\begin{abstract}
	We give three conditions on initial data for the blowing up of the corresponding solutions to some system of Klein-Gordon equations on the three dimensional Euclidean space. We first use Levine's concavity argument to show that the negativeness of energy leads to the blowing up of local solutions in finite time. For the data of positive energy, we give a sufficient condition so that the corresponding solution blows up in finite time. This condition embodies datum with arbitrarily large energy. At last we use Payne-Sattinger's potential well argument to classify the datum with energy not so large (to be exact, below the ground states) into two parts: one part consists of datum leading to blowing-up solutions in finite time, while the other part consists of datum that leads to the global solutions.
\end{abstract}
\maketitle	


\section{Introduction}
In this article, we are going to give {several }criteria for finite time blow up of solutions to the following system of Klein-Gordon equations
\begin{equation}\label{eq:skg:int}
\left\{
\begin{split}
 u_{tt}-\Delta u+u&=u^3+\beta v^2u, \mbox { in } \mathbb{R}\times\mathbb{R}^3,\\
 v_{tt}-\Delta v+v&=v^3+\beta u^2v, \mbox { in } \mathbb{R}\times\mathbb{R}^3,\\
(u(0),u_t(0))&=(u_0,u_1),\mbox { in } \mathbb{R}^3, \\
 (v(0),v_t(0))&=(v_0,v_1),\mbox { in }\mathbb{R}^3,
\end{split}
\right.
\end{equation}	
where $\beta\in{\mathbb{R}}$ is a parameter to be specified in each criterion.

This system was introduced by Segal in \cite{segal1965} to model the motion of charged mesons in electromagnetic field. Similar systems of coupled wave and Klein-Gordon equations are also proposed to model some physical phenomena such as the interaction of mass and massless classical fields, and long longitudinal waves in elastic bi-layers, see \cite{georgiev1990,  ionescu2019, 2007Coupled} and references therein for more details.

For the mathematical study, we can apply a scaling argument to the corresponding homogeneous system, to see that the system \eqref{eq:skg:int} falls into the energy subcritical class. Hence, we can apply the Banach fixed point theorem in the energy space $H^1\times H^1$ to construct a local solution $\left(u(t),v(t)\right)$ for $t\in[0,T)$ with $T$ small ({see Section \ref{sec:pre} for details}). As long as the solution $\left(u(t),v(t)\right)$ exists, the energy functional
\begin{multline}\label{energy}
		E(t):=E[u(t),\partial_tu(t);v(t),\partial_tv(t)]=\frac{1}{2}\int_{\mathbb{R}^3} \left[|\nabla u|^2+|u|^2+|u_t|^2+|\nabla v|^2+|v|^2+|v_t|^2\right]\mathrm{d}x\\
 -\frac{1}{4}\int_{\mathbb{R}^3}\left(u^4+v^4+2\beta u^2v^2\right)\mathrm{d}x
\end{multline}
remains invariant as $t$ varies. However, as $t$ varies, the kinetic energy (the first term) and the potential energy (the second term) might become arbitrarily large, so that we can not repeat the fixed point argument to extend the local solution to a global one. Even worse the solution might exhibit the blow up phenomenon. Here by the blow up of solutions, we mean that its associated `mass'
\begin{equation}\label{mass}
	y(t):=||(u(t),v(t))||_{L^2\times L^2}^2=||u||_{L^2}^2+||v||_{L^2}^2
\end{equation}
is going to become infinite in finite time. In this article, we will give some sufficient conditions on the initial datum for such a phenomenon to occur.

\subsection{Negative energy regime}\label{subsec:1.1}
The first one we are going to give falls into the philosophy that `negative energy leads to blow up in finite time', as is stated
\begin{theorem}\label{thm:neg:int}
	Fix $\beta\in [-1,\infty)$. Given the initial data $\left((u_0,u_1),(v_0,v_1)\right)\in \mathcal{H}^2$ with $\mathcal{H}:=H^1\times L^2$, let $\left(u(t),v(t)\right)$ be the corresponding solution to \eqref{eq:skg:int}. Assume either of the following holds:
	\begin{enumerate}[(i)]
		\item  the negative energy condition $E(0)<0$;
		\item  the zero energy conditon $E(0)=0$, together with {the angle condition} that $$\int_{\mathbb{R}^3}\left[u_0u_1+v_0v_1\right]\mathrm{d}x>0.$$
	\end{enumerate}
	Then, $\left(u(t),v(t)\right)$ blows up in finite time.
\end{theorem}
{\begin{remark}\label{rmk:01080}
		Comparing with the forthcoming $\mathbf{Theorem}$ \ref{thm:payne-sattinger:int}, one can find the range of $\beta$ is much larger in $\mathbf{Theorem}$ \ref{thm:neg:int}, and that for $\beta$ in the common range $[0,\infty)$,  $\mathbf{Theorem}$ \ref{thm:payne-sattinger:int} implies  $\mathbf{Theorem}$ \ref{thm:neg:int}. But for $\beta\in[-1,0)$, we would not expect a result as that in $\mathbf{Theorem}$ \ref{thm:payne-sattinger:int}, see $\mathbf{Remark}$ \ref{rmk:0108} for the technical explanation. Therefore, we shall give the full proof of $\mathbf{Theorem}$ \ref{thm:neg:int} in an independent section. One can also find that the proof, especially $\mathbf{Lemma}$ \ref{lem:nehari:neg}, is of independent interest.
	\end{remark}

In order to see how the negativeness of initial energy is used, let's briefly explain the key idea in the proof of \textit{(i)} in Theorem \ref{thm:neg:int}. We argue by contradiction, by assuming that the solution exists globally. Fix $\gamma\in(0,1)$. Differentiating $y(t)$ and doing some algebraic calculations (indeed, this is completed by combining \eqref{eq:ene:d2} and \eqref{eq:2000}), we obtain
\begin{equation}\label{eq:5001}
		\frac{\mathrm{d}^2}{\mathrm{d}t^2}y(t)\geq (4+2\gamma)\|(u_t,v_t)\|^2_{L^2\times L^2}+2\gamma \left(\|(u,v)\|^2_{H^1\times H^1}\right)-4(1+\gamma)E(t)\geq -4(1+\gamma)E(t).
\end{equation}
From this, using the conservation of energy and doing some differential calculations, it follows
\begin{equation}\label{eq:5002}
	\frac{\mathrm{d}^2}{\mathrm{d}t^2}y(t)\geq -4(1+\gamma)E(0).
\end{equation}
Since $-4(1+\gamma)E(0)$ is strictly positive, we infer from the fundamental theorem of calculus that there exists $t_b>0$ such that $\frac{\mathrm{d}}{\mathrm{d}t}y(t_b)>0$ and $y(t_b)>0$. Then we can apply Levine's lemma (see Lemma \ref{lem:levine:1}) with $\psi=y$ to see that $y$ blows up in finite time.

We see that the negativeness of initial energy, allows us to conclude from \eqref{eq:5002} that $y(t)$ and its derivative $\frac{\mathrm{d}}{\mathrm{d}t}y(t)$ will become strictly positive after some finite time, which leads to blowing up. However, if the energy is positive, we can not use this argument directly. Fortunately, we can use the projection functional and Payne-Sattinger well argument to give two sufficient conditions for finite time blow up, which contains datum of positive energy. These will serve as our second and third criteria.
}

\subsection{Criterion in terms of the projection functional}
As is alluded to above, our second criterion will allow the energy to be positive and the angle condition to be non-negative, but at the cost that the initial energy is under control of $y(t)$ and the projection functional $P(t)$ defined for the solution $\left(u(t),v(t)\right)$ to \eqref{eq:skg:int} by
\begin{equation}\label{eq:proj:int}
	P(t):=\begin{cases}
	\frac{\left\langle\big(u(t),v(t)\big),\big(u_t(t),v_t(t)\big)\right\rangle^2_{(L^2\times L^2)^2}}{||(u(t),v(t))||_{L^2\times L^2}^2}, \quad \mathrm{if}\ ||(u(t),v(t))||_{L^2\times L^2}\not=0;\\
	0,\quad\ \ \ \ \ \ \ \ \ \ \ \ \ \ \ \ \ \ \ \ \ \ \ \ \ \ \ \ \ \ \ \ \ \  \mathrm{if}\ ||(u(t),v(t))||_{L^2\times L^2}=0.
	\end{cases}
\end{equation}
 More precisely, this is stated as
\begin{theorem}\label{thm:skg:int}
	Fix $\beta\in\mathbb{R}$. Let $\left((u_0,u_1),(v_0,v_1)\right)\in \mathcal{H}^2$ and $\left(u(t),v(t)\right)$ its corresponding solution to \eqref{eq:skg:int}. Assume that $\left(u(t),v(t)\right)$ satisfies
	\begin{enumerate}[(i)]
		\item $\left\|\big(u_0,v_0\big)\right\|_{L^2\times L^2}>0$;
		\item $\big\langle(u_0,v_0),(u_1,v_1)\big\rangle_{(L^2\times L^2)^2}\geq0$;
		\item $\frac{1}{4}y(0)+\frac{1}{2}P(0)\geq E(0)>0$.
	\end{enumerate}
	Then the solution $\left(u(t),v(t)\right)$ blows up in finite time.
\end{theorem}	
{
\begin{remark}
	The projection functional $P$ is indeed a two dimensional analogue of the scalar version in \cite{kutv2016}, where Kutev, Kolkovska and Dimova proved the same result for Klein-Gordon equation with more general nonlinearities. One can also find the motivation for introducing the scalar version of the functional $P(t)$ in \cite{kutv2016}.
\end{remark}
\begin{remark}
	The advantage of Theorem \ref{thm:skg:int} is that the initial energy can be arbitrarily large, in particular can be larger than that of ground states. Indeed, there have been results of blowing up for solutions with energy above the ground states, see \cite{nakanishischlag2011jde,nakanishischlag2012}. What's more, in these two works, Nakanishi and Schlag gave the qualitative classification of solutions with energy slightly above the ground states.
	
	In our case, datum with arbitrarily large energy, that leads to blowing up solution, is constructed in Section \ref{sec:diff}.
\end{remark}

As mentioned at the end of Subsection \ref{subsec:1.1}, the key point in the proof of Theorem \ref{thm:skg:int} is to establish a differential inequality as \eqref{eq:5002}. We also argue by contradiction, assuming the solution does not blow up in finite time. In order to derive a contradiction, we carry out the following strategy.
\begin{enumerate}
	\item We first show that $K(t)<0$ for all $t\geq0$ (see Lemma \ref{lem:skg:1}). This is proved by a contradiction argument. By the assumption we get $K(0)<0$. Since $K(t)$ is continuous in $t$ (see Proposition \ref{prop:cauchy}), we can assume $t_0>0$ is the first time at which $K$ vanishes. By Lemma \ref{lem:skg}, $y(t)$ and $\frac{\mathrm{d}}{\mathrm{d}t}y(t)$ are strictly increasing on $[0,t_0]$, and hence $P(t)$ is well defined and is strictly increasing on $[0,t_0]$. We can then rewrite $K(t_0)$ in terms of $E(t_0)$ (see \eqref{eq:equa:1}) to obtain by using the energy conservation
	\begin{equation}\label{eq:0726}
		\frac{1}{4}K(t_0)\leq \frac{1}{4}\left(y(0)-y(t_0)\right)+\frac{1}{2}\left(P(0)-P(t_0)\right).
	\end{equation}
	This is strictly negative and hence $0=\frac{1}{4}K(t_0)<0$, a contradiction.
	\item Using fundamental theorem of calculus and the monotonicity of $\frac{\mathrm{d}}{\mathrm{d}t}y(t)$, we can show $$2P(0)+y(0)-y(t)\leq \frac{1}{2}\left(\frac{\mathrm{d}}{\mathrm{d}t}y(0)\right)^2-t\frac{\mathrm{d}}{\mathrm{d}t}y(0)$$ which is negative for all $t>t_b=\frac{\frac{\mathrm{d}}{\mathrm{d}t}y(0)}{2}$ (see Lemma \ref{lem:skg}).
	\item By rewriting $K(t)$ in terms of $E(t)$ and using the conservation of energy once again (see \eqref{eq:equa:1}), we get
	\begin{equation}
		K(t)\leq 2P(0) +y(0)-y(t)-2\left\|(u_t(t),v_t(t))\right\|^2_{L^2\times L^2},\ \ \forall t\geq 0.
	\end{equation}
	Combining this with \eqref{eq:0726}, we obtain
	\begin{equation}\label{eq:07260}
		K(t)\leq -2\left\|(u_t(t),v_t(t))\right\|^2_{L^2\times L^2},\ \ \forall t\geq t_b.
	\end{equation}
	\item Substituting \eqref{eq:07260} into \eqref{eq:diff:1}, we get
	\begin{equation}\label{eq:07261}
			\frac{\mathrm{d}^2}{\mathrm{d}t^2}y(t)\geq 2\|(u_t,v_t)\|_{L^2\times L^2}^2+4\|(u_t,v_t)\|_{L^2\times L^2}^2=6\|(u_t,v_t)\|_{L^2\times L^2}^2\ \ \forall t\geq t_b.
	\end{equation}
	It then follows from Cauchy-Schwarz inequality that
	\begin{equation}
			y(t) \frac{\mathrm{d}^2}{\mathrm{d}t^2}y(t)-\frac{3}{2}\left(\frac{\mathrm{d}}{\mathrm{d}t}y(t)\right)^2 \geq 0.
	\end{equation}
We can use Levine's concavity argument (Lemma \ref{lem:levine:1}) to see that the solution blows up in finite time.
\end{enumerate}
}

\subsection{Criterion in the spirit of Payne and Sattinger}
{Our third criterion will be given with the help of the stationary solutions $\left(\phi,\psi\right)$ to \eqref{eq:skg:int}, that is $\left(\phi,\psi\right)$ solves the equation}
\begin{equation}\label{eq:SKG:sta:int}
\left\{
\begin{split}
-\Delta \phi+\phi&=&\phi^3+\beta\psi^2\phi,\\
-\Delta \psi+\psi&=&\psi^3+\beta\phi^2\psi.
\end{split}
\right.
\end{equation}
This equation is indeed the Euler-Lagrangian equation of the functional
\begin{equation}\label{eq:J:int}
J[\phi,\psi]:=\frac{1}{2}\int_{\mathbb{R}^3}\left[|\nabla\phi|^2+|\phi|^2+|\nabla\psi|^2+|\psi|^2\right]\mathrm{d}x -\frac{1}{4}\int_{\mathbb{R}^3}\left[|\phi|^4+|\psi|^4+2\beta\phi^2\psi^2\right]\mathrm{d}x.
\end{equation}
By \textbf{Lemma \ref{lem:mount:pass}}, the functional $J[\phi,\psi]$ enjoys the mountain pass geometry. Consequently for each $\left(\phi,\psi\right)\in H^1\times H^1\backslash\{(0,0)\}$,
$
J[\lambda\phi,\lambda\psi]
$
attains its positive maximum at $\lambda=\lambda_\ast$ for some number $\lambda_{\ast}>0$. This suggests us to renormalize $\left(\phi,\psi\right)$ so that $\lambda_{\ast}=1$ and hence we have
\begin{equation*}
\frac{\mathrm{d}}{\mathrm{d}\lambda} \left[J[\lambda\phi,\lambda\psi]\right]\big|_{\lambda=\lambda_{\ast}}=0.
\end{equation*}
The left hand side will be denoted by $K[\phi,\psi]$ and it has the expression
\begin{equation}\label{eq:3000}
	K[\phi,\psi]=\int_{\mathbb{R}^3}\left[|\nabla\phi|^2+|\phi|^2+|\nabla\psi|^2+|\psi|^2\right]\mathrm{d}x -\int_{\mathbb{R}^3}\left[|\phi|^4+|\psi|^4+2\beta\phi^2\psi^2\right]\mathrm{d}x.
\end{equation}
If restricting $\beta\geq 0$, then it follows from {\textbf{Proposition 3.5} in \cite{sirakov07}} that \eqref{eq:SKG:sta:int} has a non-trivial  solution $\left(\tilde{\phi},\tilde{\psi}\right)$, each component of which is non-negative, spherically symmetric and non-increasing in radial direction. Consequently the lowest mountain pass level defined by
\begin{equation*}
d:=\min\left\{J[\phi,\psi]\ |\ K[\phi,\psi]=0,(\phi,\psi)\neq(0,0)\right\}=J[\tilde\phi,\tilde\psi]
\end{equation*}
equaling to
\begin{equation*}
\frac{1}{4}\int_{\mathbb{R}^3}|\tilde \phi|^4+|\tilde \psi|^4+2\beta\tilde{\phi}^2\tilde{\psi}^2\mathrm{d}x,
\end{equation*}
{is strictly positive}.
We use $d$ to define the following two subsets of initial datum:
\begin{equation*}
\mathcal{W}:=\left\{ \left((u_0,u_1),(v_0,v_1)\right)\in\mathcal{H}\times\mathcal{H}\big| E[u_0,u_1;v_0,v_1]<d, K[u_0,v_0]\geq0 \right\},
\end{equation*}
and
\begin{equation*}
\mathcal{E}:=\left\{ \left((u_0,u_1),(v_0,v_1)\right)\in\mathcal{H}\times\mathcal{H}\big| E[u_0,u_1;v_0,v_1]<d, K[u_0,v_0]<0 \right\},
\end{equation*}
where $\mathcal{H}:=H^1\times L^2$.

We are now in a position to state {our third criterion.}
\begin{theorem}\label{thm:payne-sattinger:int}Fix $\beta\in [0,\infty)$. Let $\left((u_0,u_1),(v_0,v_1)\right)\in \mathcal{H}\times \mathcal{H}$ and $(u(t),v(t))$ the corresponding local solution to \eqref{eq:skg:int}. Then we have the following dichotomy
	\begin{enumerate}[(i)]
		\item if $\left((u_0,u_1),(v_0,v_1)\right)\in\mathcal{W}$, $(u(t),v(t))$ can be extended to a global solution;
		\item if $\left((u_0,u_1),(v_0,v_1)\right)\in\mathcal{E}$, $(u(t),v(t))$ is going to blow up in finite time.
	\end{enumerate}
\end{theorem}

\begin{remark}
	It is easy to extend results in $\mathbf{Theorem}$ \ref{thm:payne-sattinger:int} to the following general system
	\begin{equation}\label{eq:skg:int0}
	\left\{
	\begin{split}
	u_{tt}-\Delta u+ u&=\mu_1u^3+\beta v^2u, \mbox { in } \mathbb{R}\times\mathbb{R}^3,\\
	v_{tt}-\Delta v+ v&=\mu_2v^3+\beta u^2v, \mbox { in } \mathbb{R}\times\mathbb{R}^3,
	\end{split}
	\right.
	\end{equation}	
	where $\mu_1,\mu_2>0$ and $\beta\geq0$. Since the corresponding system of elliptic equations (see \cite[Proposition 3.5]{sirakov07}) has the lowest mountain pass level $d$ being positive and possesses a constrained characterization of $d$, one has no difficulty in going through our proof for \eqref{eq:skg:int0}.
	
	As far as the authors know, the more general systems of elliptic equations like \eqref{eq:SKG:sta:int} have been studied in \cite{chenzou2012,chenzou2015,pengwangwang2019,soavetavares2016,tavares2011} and \cite{weiwu2020}, where nontrivial solutions are proven to exist. Thus as long as one can show $d$ has a constrained characterization, we might expect to extend the results in $\mathbf{Theorem}$ \ref{thm:payne-sattinger:int} to the corresponding system of Klein-Gordon equations, by going through our argument in $\mathrm{Section}$ \ref{sect:proof:thm3}.
\end{remark}
\begin{remark}
	The result in {$\mathbf{Theorem}$} \ref{thm:payne-sattinger:int} also holds for a modification of \eqref{eq:skg:int} by adding the global damping, see \cite{wu2012}. See also \cite{xu2010} for results about the scalar equation with damping term.
\end{remark}
\begin{remark}
	By combining Payne-Sattinger's potential well argument with Kenig-Merle's compactness-rigidity argument, Xu \cite{xu2013} classified the solutions to some coupled Schr\"{o}dinger system with energy below the ground states.
\end{remark}
{
In order to prove \textbf{Theorem} \ref{thm:payne-sattinger:int}, we shall give a new characterization of $d$.  We define the functional
\begin{equation*}
	E^1[u_1,u_2]:= \frac{1}{4}\left(\|u_1\|_{H^1}^2+\|u_2\|_{H^1}^2\right)
\end{equation*}
on $H^1\times H^1$ and denote
\begin{equation}\label{Np:intr}
	\mathcal{N}':=\left\{(u_1,u_2)\in H^1\times H^1|
	(u_1,u_2)\not= (0,0),
	K[u_1,u_2]\leq 0   \right\}.
\end{equation}
\begin{theorem}\label{thm:char:d:int}
	Fix $\beta\geq 0$. Then the quantity $\mathop{\inf}\limits_{(u_1,u_2)\in \mathcal{N}'} E^1[u_1,u_2]$ is attained by some element in $\mathcal{N}'$ and it equals $d$.
\end{theorem}
\begin{remark}
	$\mathrm{\mathbf{Theorem}}$\ref{thm:char:d:int} is indeed a generalization of \textbf{Theorem 3.5} in \cite{sirakov07}, where $\mathcal{N}'$ is replaced by the Nehari manifold. It is as well a generalization of the scalar version{, see \cite{nakanishischlag2011}}.
\end{remark}
\begin{remark}\label{rmk:0108}
	As is mentioned in $\mathbf{Remark}$ \ref{rmk:01080}, we can have results in $\mathbf{Theorem}$ \ref{thm:payne-sattinger:int} only for $\beta\in[0,\infty)$. This is because we prove this theorem by using $\mathbf{Theorem}$ \ref{thm:char:d:int}, which is stated only for $\beta\geq 0$. In fact, we can not lower down the value of $\beta$ strictly below zero by only using our argument. Therefore, it is curious for us whether or not the result in $\mathbf{Theorem}$ \ref{thm:char:d:int} holds for some $\beta<0$.
\end{remark}
An immediate corollary of Theorem \ref{thm:char:d:int}, is the following minimal property of $d$.
\begin{corollary}\label{lem:tech}Fix $\beta\geq0$.
	Let $\{\left(u_n,v_n\right)\}_{n\geq 1}$ be a sequence in $H^1\times H^1$ satisfy
	\begin{enumerate}[(i)]
		\item $K[u_n,v_n]<0$;
		\item $K[u_n,v_n]$ tends to $0$ as $n$ tends to infinity.
	\end{enumerate}
	Then if $\left\|\left(u_n,v_n\right)\right\|_{H^1\times H^1}\neq 0$ for each $n$, we have
	\begin{equation}\label{eq:1600}
		\liminf_{n\rightarrow\infty} J[u_n,v_n]\geq d.
	\end{equation}
\end{corollary}

\begin{proof}[Proof of Corollary \ref{lem:tech}]
	By definitions of $K$ and $J$, we have for each $n$,
	\begin{equation}\label{eq:2505}
		J[u_n,v_n]=\frac{1}{4}K[u_n,v_n]+E^1[u_n,v_n].
	\end{equation}
	Sending $n$ to infinity and using the assumption $K[u_n,v_n]\rightarrow_{n\rightarrow\infty}0$, we get
	\begin{equation}\label{eq:4000}
		\liminf_{n\rightarrow\infty}J[u_n,v_n]=\lim\inf_{n\rightarrow\infty}E^1[u_n,v_n].
	\end{equation}
	Since $\left\|\left(u_n,v_n\right)\right\|_{H^1\times H^1}\neq 0$ and $K[u_n,v_n]<0$ for each $n$, we have $(u_n,v_n)\in\mathcal{N}'$. Consequently by Theorem \ref{thm:char:d:int}, we have
	\begin{equation}\label{eq:4001}
		\lim\inf_{n\rightarrow\infty}E^1[u_n,v_n]\geq d.
	\end{equation}
	Combining the inequality above with \eqref{eq:4000}, we get \eqref{eq:1600}. This completes the proof.
\end{proof}

We now sketch the proof of \textbf{Theorem} \ref{thm:payne-sattinger:int}. As in \cite{paynesattinger}, we argue by contradiction, assuming that the solution exists for all time. Using Corollary \ref{lem:tech}, we can show $\mathcal{E}$ is invariant under the flow of \eqref{eq:skg:int}. It then follows that the differential expression
\begin{equation} \frac{\mathrm{d}^2}{\mathrm{d}t^2}y(t)=2\int_{\mathbb{R}^3}\left(|u_t|^2+|v_t|^2\right)\mathrm{d}x-2K[u,v](t)
\end{equation}
is strictly positive and hence $y(t)$ is convex on $[0,\infty)$.  By using Corollary \ref{lem:tech} once again combining with Claim \ref{claim:5}, we show that there exist a large number $t_{\ast}>0$ so that $y(t)>0$ for $t\geq t_{\ast}$. This together with the strict convexity of $y$ implies that $y(t)\rightarrow_{t\rightarrow\infty}+\infty$, in particular there exists a large number $t_{\ast\ast}>0$ so that
\begin{equation}\label{eq:7000}
	2y(t)-\max\left(0,8E(0)\right)>\delta,\ \ \forall t\geq t_{\ast\ast}.
\end{equation}
for some $\delta>0$. By rewriting the expression of $\frac{\mathrm{d}^2}{\mathrm{d}t^2}y(t)$ in terms of $E(t)$ and using the conservation of energy, we get
\begin{equation}\label{eq:7001}
	\frac{\mathrm{d}^2}{\mathrm{d}t^2}y(t)\geq \frac{3}{2}\frac{(\frac{\mathrm{d}}{\mathrm{d}t}y(t))^2}{y(t)}+2y(t)-8E(0),\ \ \forall t\geq t_{\ast\ast}.
\end{equation}
Substituting \eqref{eq:7000} into \eqref{eq:7001}, we get an inequality in analogue of \eqref{eq:5001}. Consequently we can use Levine's concavity argument to get the contradiction, which will close our argument.}

%

We end this introductory part by describing the organization of this article. In Section \ref{sec:pre}, after giving some basic notations, we will establish the local Cauchy theory for \eqref{eq:skg:int} in the energy space. In Section \ref{sec:proof:thm1}, we will use Levine's concavity argument (see Lemma \ref{lem:levine:1}) to prove Theorem \ref{thm:neg:int}. In Section \ref{sec:proof:thm2}, we will prove Theorem \ref{thm:skg:int}. In Section \ref{sect:proof:thm3}, we will first show that \eqref{eq:skg:int} has nontrivial stationary solutions, which we will use to define the lowest mountain pass level $d$; next we will give a new characterization of $d$; using this new characterization, we will prove Theorem \ref{thm:payne-sattinger:int}. In the last section, we will construct datum that satisfy each assumption in \textbf{Theorem} \ref{thm:neg:int}.
\section{Preliminaries}
\label{sec:pre}
We devote this section to specifying some basic notations, with which we are going to work, and some basic results, which are basis for our forthcoming analysis.

{Let $L^2$ be the space of square integrable real-valued functions defined on $\mathbb{R}^3$ and $H^1$ the subspace of $L^2$, consisting of elements whose gradients are also square integrable. Denote $\mathcal{H}:=H^1\times L^2$. We equip the inner product on the product space $\mathcal{H}\times \mathcal{H}$ by setting
\begin{equation*}
	\left\langle\left((u,v),(\tilde{u},\tilde{v})\right)\right\rangle_{\mathcal{H}\times\mathcal{H}}:=\langle u,\tilde{u}\rangle_{H^1\times H^1}+\langle v,\tilde{v}\rangle_{L^2\times L^2}.
\end{equation*}
Denote by $\|\cdot\|_{\mathcal{H}\times\mathcal{H}}$ the norm induced by this inner product.} More generally, for an Hilbert space $\mathfrak{h}$, we will denote by $\langle\cdot,\cdot\rangle_{\mathfrak{h}}$ the inner product on it.

\vspace*{2pt}
We first solve \eqref{eq:skg:int}.
\begin{proposition}\label{prop:cauchy}
	 Fix $\beta\in\mathbb{R}$. For $\left((u_0,u_1),(v_0,v_1)\right)\in \mathcal{H}\times \mathcal{H}$, there exists a maximal $T_{\max}>0$ or $T_{max}=+\infty$, so that \eqref{eq:skg:int} has a unique solution $\left(u(t),v(t)\right)$ on $[0,T_{max})$ satisfying
	 \begin{equation*}
	 	\left(\left(u(t),\partial_tu(t)\right),\left(v(t),\partial_tv(t)\right)\right)\in L^\infty([0,T_{\max});\mathcal{H}\times\mathcal{H}).
	 \end{equation*}
 Moreover, the `mass' functional $y(t)$ is differentiable of the first and the second order for almost every $t\in [0,T_{max})$.
\end{proposition}
\begin{proof}
	We rewrite \eqref{eq:skg:int} in the integral vectorial form
	\begin{equation}\label{eq:skg:inte}
		\Lambda[u,v](t):=\begin{bmatrix}
			\cos\left(t\sqrt{-\Delta}\right)u_0+\frac{\sin\left(t\sqrt{-\Delta}\right)}{\sqrt{-\Delta}}u_1 \\
			\cos\left(t\sqrt{-\Delta}\right)v_0+\frac{\sin\left(t\sqrt{-\Delta}\right)}{\sqrt{-\Delta}}v_1
		\end{bmatrix} +\mathcal{N}[u,v](t)
	\end{equation}
	where $\mathcal{N}[u,v](t)$ is defined as
	\begin{equation}
		\mathcal{N}[u,v]:=\begin{bmatrix}
			\int_0^t\frac{\sin\left((t-t')\sqrt{-\Delta}\right)}{\sqrt{-\Delta}}\left(u^3(t')+\beta v^2(t')u(t')\right)\mathrm{d}t'\\
			\int_0^t\frac{\sin\left((t-t')\sqrt{-\Delta}\right)}{\sqrt{-\Delta}}\left(v^3(t')+\beta v(t')u^2(t')\right)\mathrm{d}t'
		\end{bmatrix}
	\end{equation}
	Taking the $H^1\times H^1$-norm on both sides of \eqref{eq:skg:inte}, we get for each $T>0$
	\begin{equation}\label{eq:1701}
		\left\|(u,v)\right\|_{L^\infty\left([0,T];H^1\times H^1\right)}\leq \left\|(u_0,u_1)\right\|_{H^1\times L^2}+\left\|(v_0,v_1)\right\|_{H^1\times L^2}+\left\|\mathcal{N}[u,v]\right\|_{L^\infty\left([0,T];H^1\times H^1\right)}
	\end{equation}
	Applying H\"{o}lder inequality and then the Sobolev embedding $H^1\subset L^6$, we can bound $\left\|\mathcal{N}[u,v]\right\|_{L^\infty\left([0,T];H^1\times H^1\right)}$ by
	\begin{equation}\label{eq:1700}
		\left\|\mathcal{N}[u,v]\right\|_{L^\infty\left([0,T];H^1\times H^1\right)}\leq CT\left\|(u,v)\right\|^3_{L^\infty\left([0,T],H^1\times H^1\right)}.
	\end{equation}
	Here $C$ is a positive constant depending on the Sobolev embedding and $\beta$. If we take $R=2\left\|(u_0,u_1)\right\|_{H^1\times L^2}+2\left\|(v_0,v_1)\right\|_{H^1\times L^2}$, then for $T=\frac{1}{16CR^2}$, we can see from \eqref{eq:1701} and \eqref{eq:1700} that $\Lambda$ maps the closed ball $B(0,2R)\subset C([0,T],H^1\times H^1)$ into itself.  Next letting $(u,v),(\tilde{u},\tilde{v})\in B(0,2R)$, we observe that
	\begin{equation}\label{eq:1704}
		\left\|\Lambda[u,v]-\Lambda[\tilde{u},\tilde{v}]\right\|_{L^\infty([0,T],H^1\times H^1)}=\left\|\mathcal{N}[u,v]-\mathcal{N}[\tilde{u},\tilde{v}]\right\|_{L^\infty([0,T],H^1\times H^1)}.
	\end{equation}
	Inserting the following algebraic relations
	\begin{equation*}
		u^3-\tilde{u}^3=(u-\tilde{u})\left(u^2+u\tilde{u}+\tilde{u}^2\right)
	\end{equation*}
	and
	\begin{equation*}
		uv^2-\tilde{u}\tilde{v}^2=(u-\tilde{u})v^2+\tilde{u}(v-\tilde{v})(v+\tilde{v})
	\end{equation*}
	into the first component of $\mathcal{N}[u,v]$, and using H\"{o}lder's inequality and Sobolev embedding $H^1\subset L^6$, we see that the first component $\mathcal{N}_1[u,v]-\mathcal{N}_1[\tilde{u},\tilde{v}]$ of $\mathcal{N}[u,v]-\mathcal{N}[\tilde{u},\tilde{v}]$ is under control
	\begin{equation}\label{eq:1702}
		\left\|\mathcal{N}_1[u,v]-\mathcal{N}_1[\tilde{u},\tilde{v}]\right\|_{L^\infty([0,T],H^1)}\leq CT\left\|u-\tilde{u}\right\|_{L^\infty([0,T],H^1)}\left(\|u\|^2_{L^\infty([0,T],H^1)}+\|\tilde{u}\|^2_{L^\infty([0,T],H^1)}\right)
	\end{equation}
	\begin{equation*}
		\ \ \ \ \ \ \ \ \ \ \ \ \ \ \ \ +CT\left\|u-\tilde{u}\right\|_{L^\infty([0,T],H^1)}\|v\|^2_{L^\infty([0,T],H^1)}
	\end{equation*}
	\begin{equation*}
		\ \ \ \ \ \ \ \ \ \ \ \ \ \ \ \ \
		+ CT\left\|u-\tilde{u}\right\|_{L^\infty([0,T],H^1)}\times \|\tilde{u}\|_{L^\infty([0,T],H^1)}\left(\|v\|_{L^\infty([0,T],H^1)}+\|\tilde{v}\|_{L^\infty([0,T],H^1)}\right)
	\end{equation*}
	for some positive constant $C$.
	
	Inserting the following algebraic relations
	\begin{equation*}
		v^3-\tilde{v}^3=(v-\tilde{v})\left(v^2+v\tilde{v}+\tilde{v}^2\right)
	\end{equation*}
	and
	\begin{equation*}
		vu^2-\tilde{v}\tilde{u}^2=(v-\tilde{v})u^2+\tilde{v}(u-\tilde{u})(u+\tilde{u})
	\end{equation*}
	into the second component of $\mathcal{N}[u,v]$, and using the same argument, we obtain
	\begin{equation}\label{eq:1703}
		\left\|\mathcal{N}_2[u,v]-\mathcal{N}_2[\tilde{u},\tilde{v}]\right\|_{L^\infty([0,T],H^1)}\leq CT\left\|v-\tilde{v}\right\|_{L^\infty([0,T],H^1)}\left(\|v\|^2_{L^\infty([0,T],H^1)}+\|\tilde{v}\|^2_{L^\infty([0,T],H^1)}\right)
	\end{equation}
	\begin{equation*}
		\ \ \ \ \ \ \ \ \ \ \ \ \ \ \ \ +CT\left\|v-\tilde{v}\right\|_{L^\infty([0,T],H^1)}\|u\|^2_{L^\infty([0,T],H^1)}
	\end{equation*}
	\begin{equation*}
		\ \ \ \ \ \ \ \ \ \ \ \ \ \ \ \ \
		+ CT\left\|v-\tilde{v}\right\|_{L^\infty([0,T],H^1)}\times \|\tilde{v}\|_{L^\infty([0,T],H^1)}\left(\|u\|_{L^\infty([0,T],H^1)}+\|\tilde{u}\|_{L^\infty([0,T],H^1)}\right)
	\end{equation*}
	for some positive constant $C$.
	
	Substituting \eqref{eq:1702} and \eqref{eq:1703} into \eqref{eq:1704}, and using Young's inequality, we get
	\begin{equation}
        \begin{split}
		&\left\|\Lambda[u,v]-\Lambda[\tilde{u},\tilde{v}]\right\|_{L^\infty([0,T],H^1\times H^1)} \\
~~~~~~~~~~~~~  & \leq CT\left\|(u,v)-(\tilde{u},\tilde{v})\right\|_{L^\infty([0,T],H^1\times H^1)}\left(\|(u,v)\|^2_{L^\infty([0,T],H^1\times H^1)}+\|(\tilde{u},\tilde{v})\|^2_{L^\infty([0,T],H^1\times H^1)}\right)
	   \end{split}
    \end{equation}
	for some constant $C$. If we take {$T\leq\frac{1}{16CR^2}$}, we see that $\Lambda:B(0,2R)\rightarrow B(0,2R)$ is a contraction map.
	
	Since $B(0,2R)$ is closed, then we conclude from Banach fixed theorem that there exist a solution $\left(u(t),v(t)\right)$ in $L^\infty([0,T],H^1\times H^1)$ to \eqref{eq:skg:int}. Up to now, we have constructed a unique solution $\left(u(t),v(t)\right)$ in a small time interval. We next try to extend this solution onto a maximal time interval.
	
	{
	At time $T$, as long as  $\left\|\left((u(T),\partial_tu(T)),(v(T),\partial_tv(T))\right)\right\|_{\mathcal{H}\times \mathcal{H}}$ is finite, we can repeat the Banach fixed point argument, beginning at the time $T$, to extend the solution to live on a longer time interval, say $[0, T+T_1]$ for some $T_1>0$ depending on $\left\|\left((u(T),\partial_tu(T)),(v(T),\partial_tv(T))\right)\right\|_{\mathcal{H}\times \mathcal{H}}$. Repeating this process, we can extend the solution to live on the maximal time interval $[0,T_{max})$.	
		
	From the proof of the existence of local solutions, we know that the solution $\left(u(t),v(t)\right)$ has its derivatives $\left(\partial_tu,\partial_tv\right)\in L^\infty([0,T_{max});L^2\times L^2)$. Hence $y(t)$ is differentiable for almost every $t\in[0,T_{max})$. If we use the integration by parts, we can see that $y(t)$ is also differentiable of second order, for almost every $t\in[0,T_{max})$.	
	}
\end{proof}

{We next recall two technical lemmas. The first one embodies the concavity argument, due to Levine (see \cite{levine}).}
\begin{lemma}[\cite{levine}]\label{lem:levine:1}
	Let $\psi(t)\in C^2$ for $t\geq t_b\geq 0$ and $\psi(t)>0$.
	Assume that \begin{enumerate}
		\item $\frac{\mathrm{d}^2\psi}{\mathrm{d}t^2}(t) \psi(t)-\gamma \left(\frac{\mathrm{d}\psi}{\mathrm{d}t}(t)\right)^2\geq 0, t\geq t_b, \gamma>1$
		\item $\psi(t_b)>0, \frac{\mathrm{d}\psi}{\mathrm{d}t}(t_b)>0$.
	\end{enumerate}
	Then there eixsts some finite time $t^*>0$ so that
	\begin{equation*}
		\psi(t)\rightarrow \infty,  \ \text{as}\ t\rightarrow t^*.
	\end{equation*}	
\end{lemma}	

The second one is about Schwarz spherical rearrangement {(see \cite[Appendix III.3]{BL83} and \cite[Chapter 3]{LL01} for the proof). }
\begin{lemma}\label{lem:rearrangement}
 Suppose $u_1,u_2\in H^1$ and  let $u^*_1,u^*_2$ be Schwarz spherical rearrangement of $u_1, u_2$. Then for any $p\in [2,6]$,
\begin{equation}
\|u_i^*\|_{H^1}\leq \|u_i\|_{H^1},\quad \|u_i^*\|_{L^p}{=}\|u_i\|_{L^p}, \quad \forall i\in \{1,2\}
\end{equation}
and
\begin{equation}
\quad \int_{\mathbb{R}^3}(u_1^*)^2(u_2^*)^2\mathrm{d}x\geq  \int_{\mathbb{R}^3}u_1^2u_2^2\mathrm{d}x.
\end{equation}
\end{lemma}

\section{Proof of Theorem \ref{thm:neg:int}}\label{sect:3}
\label{sec:proof:thm1}

Recall the system we are going to consider is
\begin{equation}\label{eq:skg}
	\left\{
	\begin{split}
		 u_{tt}-\Delta u+u&=u^3+\beta v^2u, \mbox { in } \mathbb{R}\times\mathbb{R}^3,\\
		 v_{tt}-\Delta v+v&=v^3+\beta u^2v, \mbox { in } \mathbb{R}\times\mathbb{R}^3,\\
		(u(0),u_t(0))&=(u_0,u_1),\mbox { in } \mathbb{R}^3, \\
		(v(0),v_t(0))&=(v_0,v_1),\mbox { in }\mathbb{R}^3.
	\end{split}
	\right.
\end{equation}	
Here $\beta\in\mathbb{R}$ is a fixed parameter.
Let $(u_0,u_1),(v_0,v_1)\in \mathcal{H}$, then we can apply Proposition \ref{prop:cauchy} to see that the equation \eqref{eq:skg} has a unique solution on the {interval $[0,T_{max})$. Our goal here is to give some conditions on the initial datum so that $T_{max}$ is finite, \textit{i.e.}, the corresponding solution blows up in finite time. }

Recall that the energy functional $E(t)$ and the `mass' functional $y(t)$ are defined respectively in \eqref{energy} and \eqref{mass}.
 By Proposition \ref{prop:cauchy}, we can differentiate $y(t)$ in $t$-variable to obtain
 \begin{equation}\label{eq:ene:d1}
 	\frac{\mathrm{d}y}{\mathrm{d}t}(t)=2\big\langle(u(t),v(t)),(u_t(t),v_t(t))\big\rangle_{\left(L^2\times L^2\right)^2},
 \end{equation}	
and
 \begin{equation}\label{eq:ene:d2}
 	\frac{\mathrm{d}^2y}{\mathrm{d}t^2}(t)=2\int_{\mathbb{R}^3}\left(|u_t|^2+|v_t|^2\right)\mathrm{d}x-2K[u,v](t),
 \end{equation}
 where  $K[u,v](t)$ is the Nehari functional defined as in \eqref{eq:3000}.

By definitions of the energy functional $E(t)=E[u,v](t)$ and the Nehari functional $K(t):=K[u,v](t)$, we can deduce by some algebraic calculations
\begin{equation}\label{eq:2000}
	2(1+\gamma)E(t)-K(t)=(1+\gamma)\left(\|(u_t,v_t)\|^2_{L^2\times L^2}\right)+\gamma \left(\|(u,v)\|^2_{H^1\times H^1}\right)+\int_{\mathbb{R}^3}\left[\frac{(1-\gamma)}{2}\left(u^4+v^4+2\beta u^2v^2\right)\right]\mathrm{d}x.
\end{equation}
Here $\gamma$ is some parameter, which we are going to specify. But for $\gamma\in (0,1]$, {if we assume $\beta\geq -1$,} the last term on the right hand side of this equality is non-negative so that we can control the kinetic energy by some proper combination of $E$ and $K$. More precisely, we have the following result.
\begin{lemma}\label{lem:nehari:neg} Fix $\beta\geq -1$.
	For each $\gamma\in (0,1]$, we then have
	\begin{equation*}
		(\gamma+1)\|(u_t,v_t)\|_{L^2\times L^2}^2+\gamma \|(u,v)\|^2_{H^1\times H^1}\leq 2(1+\gamma)E(t)-K(t),
	\end{equation*}
	for any solution $(u,v)\in H^1\times H^1$ to the equation \eqref{eq:skg:int}.
\end{lemma}	
With this lemma at hand, we are now ready to give
\begin{proof}[Proof of \textbf{Theorem \ref{thm:neg:int}}] We argue by contradiction. Assume the solution does not blow up in finite time. Then at any time $t$, applying Cauchy-Schwarz inequality in \eqref{eq:ene:d1}, we get
	\begin{equation}\label{eq:diff:2}
		\frac{\mathrm{d}}{\mathrm{d}t}y(t)\leq 2\left\|(u,v)\right\|_{L^2\times L^2}\ \left\|(u_t,v_t)\right\|_{L^2\times L^2}.
	\end{equation}
	Recall that the differential expression \eqref{eq:ene:d2} of $\frac{\mathrm{d} y}{\mathrm{d}t}$ reads
	\begin{equation}\label{eq:diff:1}
		\frac{\mathrm{d}^2}{\mathrm{d}t^2}y(t)=2\left\|(u_t,v_t)\right\|_{L^2\times L^2}^2-2K(t).
	\end{equation}
	In order to treat the term $-K(t)$, we invoke Lemma \ref{lem:nehari:neg}, to conclude that $K(t)\leq 2(1+\gamma)E(t)$ for any $\gamma \in (0,1]$. Together with the conservation of energy, we get
    \begin{equation*}
K(t)\leq 2(1+\gamma)E(0)=:-\delta <0.
    \end{equation*}
Inserting this into the above differential {equality} \eqref{eq:diff:1}, we get
	\begin{equation*}
		\frac{\mathrm{d}^2}{\mathrm{d}t^2}y(t)\geq 2\left\|(u_t,v_t)\right\|_{L^2\times L^2}^2 +2 \delta.
	\end{equation*}
	It then follows that $\frac{\mathrm{d}}{\mathrm{d}t}y(t)$ tends to $+\infty$ as $t$ approaches infinity and hence so does $y(t)$. {In order to obtain the contradiction, we use \eqref{eq:2000} and the conservation of energy to conclude from \eqref{eq:ene:d2} that
	\begin{equation}\label{eq:400}
		\frac{\mathrm{d}^2}{\mathrm{d}t^2}y(t)\geq (4+2\gamma)\|(u_t,v_t)\|^2_{L^2\times L^2}+2\gamma \left(\|(u,v)\|^2_{H^1\times H^1}\right)-4(1+\gamma)E(0).
	\end{equation}}
	By \eqref{eq:diff:2}, we bound this by Cauchy-Schwarz inequality from below
	\begin{equation*}
		\frac{\mathrm{d}^2}{\mathrm{d}t^2}y(t)\geq \left(1+\frac{\gamma}{2}\right)\frac{\left(\frac{\mathrm{d}}{\mathrm{d}t}y\right)^2}{y(t)} +2\delta.
	\end{equation*}
	We then apply Levine's lemma \ref{lem:levine:1} to conclude that $y(t)$ blows up before some finite positive time. But this contradicts the assumption that the solution exists globally. So the solution $\left(u(t),v(t)\right)$ blows up in finite time. This finishes the proof of asserted result under the condition in \textit{(i)}.
	
	Next we assume $E(0)=0$ and $\left\langle(u_0,v_0),(u_1,v_1)\right\rangle>0$. As in the first item, we argue by contradiction. Assume the solution does not blow in finite time. By \eqref{eq:400}, we have\begin{equation}\label{eq:diff:5}
		\frac{\mathrm{d}^2}{\mathrm{d}t^2}y\geq (4+2\gamma)\left(\|(u_t,v_t)\|^2_{L^2\times L^2}\right)+2\gamma \left(\|(u,v)\|^2_{H^1\times H^1}\right)\geq 0.
	\end{equation}
	Solving this differential inequality, complemented with the initial condition
    \begin{equation*}
    \frac{\mathrm{d}}{\mathrm{d}t}y(0)=2\langle(u_0,v_0),(u_1,v_1)\rangle_{(L^2\times L^2)^2}>0,
    \end{equation*}
 gives us
	\begin{equation*}
		\frac{\mathrm{d}}{\mathrm{d}t}y(t)\geq 2\big\langle(u_0,v_0),(u_1,v_1)\big\rangle_{(L^2\times L^2)^2}>0,\ \ {\forall t> 0.}
	\end{equation*}
	This infers that there exists some $t_b>0$ so that $y(t)>0$ for all $t>t_b$. This allows us to rewrite \eqref{eq:diff:5} as
	\begin{equation}
		\frac{\mathrm{d}^2}{\mathrm{d}t^2}y(t)\geq (1+\frac{\gamma}{2})\frac{\left(\frac{\mathrm{d}}{\mathrm{d}t}y\right)^2}{y(t)} +c, \forall t\in [t_b,\infty).
	\end{equation}
	Here $c$ is some positive constant, for instance we can take
    \begin{equation*}
        c=2\gamma \left(\|(u(t),v(t))\|^2_{H^1\times H^1}\right)\geq 2\gamma \left(\|(u(t_b),v(t_b))\|^2_{L^2\times L^2}\right) >0.
    \end{equation*}
	We then apply Levine's lemma \ref{lem:levine:1} to conclude that $y(t)$ blows up before some finite positive time, which contradicts the initial assumption. Thus the solution blows up in finite time under the condition in \textit{(ii)}. This completes the proof of \textbf{Theorem \ref{thm:neg:int}}.
\end{proof}

\section{Proof of Theorem \ref{thm:skg:int}}
\label{sec:proof:thm2}

We devote this section to proving \textbf{Theorem} \ref{thm:skg:int}. Let $(u_0,u_1),(v_0,v_1)\in\mathcal{H}$ and  $\left(u(t),v(t)\right), t\in[0,T_{max})$ be a solution to \eqref{eq:skg}. We define the projection functional $P(t)$ as in \eqref{eq:proj:int} and the mass as in \eqref{mass}. As long as $y(t)\neq0$, we can rewrite the energy function $E(t)$ as
\begin{equation}\displaystyle\label{eq:ene:1}
	E(t)=\frac{1}{2}(P(t)+y(t))+\frac{1}{2}\Bigg(\Bigg\|u_t-\frac{\big\langle(u(t),v(t)),(u_t(t),v_t(t))\big\rangle_{(L^2\times L^2)^2}}{||(u(t),v(t))||_{L^2\times L^2}^2}u\Bigg\|_{L^2}^2
\end{equation}
\begin{equation*}
	+\Bigg\|v_t-\frac{\big\langle(u(t),v(t)),(u_t(t),v_t(t))\big\rangle_{(L^2\times L^2)^2}}{||(u(t),v(t))||_{L^2\times L^2}^2}v\Bigg\|_{L^2}^2+||\nabla u||_{L^2}^2+||\nabla v||_{L^2}^2\Bigg)-\frac{1}{4}\int_{\mathbb{R}^3}\left(u^4+v^4+2\beta u^2v^2\right)\mathrm{d}x.
\end{equation*}

In order to prove \textbf{Theorem \ref{thm:skg:int}}, we first prove two lemmas.

\begin{lemma}\label{lem:skg}
	Under the assumption of Theorem \ref{thm:skg:int}, we assume further $K(t)<0$ for $t\in [0,T]$, $T<T_{max}$. Then
	\begin{enumerate}[(i)]
		\item {$\frac{\mathrm{d}y}{\mathrm{d}t}(t)>0, \frac{\mathrm{d} P}{\mathrm{d}t}(t)>0, \ \forall t\in (0,T]$};
		\item $y(t), \frac{\mathrm{d} y}{\mathrm{d}t}(t)$ and $P(t)$ are strictly increasing in $t\in [0,T]$;
		\item $\frac{\mathrm{d}^2y}{\mathrm{d}t^2}(t)>0$, for $t\in [0,T]$;
		\item $y(0)-y(t)+2P(0)\leq 0$, for $t\geq t_b:=\frac{1}{2}\frac{\frac{\mathrm{d} y}{\mathrm{d}t}(0)}{y(0)}$.
	\end{enumerate}
\end{lemma}
\begin{proof}
	We use \eqref{eq:ene:d2} and  the assumption $K(t)<0$ to obtain
	\begin{equation}\label{eq:300}
	\frac{\mathrm{d}^2y}{\mathrm{d}t^2}(t)=2(||u_t||_{L^2}^2+||v_t||_{L^2}^2)-2K(t)>0,
	\end{equation}
	which is just the asserted result in \textit{(iii)}.	Consequently $\frac{\mathrm{d}y}{\mathrm{d}t}(t)$ is strictly increasing on $[0,T]$. Solving this differential inequality together with the initial condition
	\begin{equation}\label{eq:301}
    \frac{\mathrm{d}y}{\mathrm{d}t}(0)=2\big\langle(u_0,v_0),(u_1,v_1)\big\rangle_{(L^2\times L^2)^2}\geq 0,
	\end{equation}
	we obtain
	\begin{equation}\label{eq:ene:d3}
		\frac{\mathrm{d}y}{\mathrm{d}t}(t)>0,\ \forall t\in (0,T].
	\end{equation}
This is indeed the first one of the stated results in \textit{(i)}. What's more, as a consequence of the fundamental theorem of calculus, $y(t)$ is increasing in $t\in [0,T]$. Combining this with the non-negativeness of $y(0)$, we also get $y(t)>0$ for $t\in(0,T]$. Thus the projection functional $P$ in \eqref{eq:proj:int} is well defined for $t>0$.

Differentiating $P$ in the time variable $t$ and using \eqref{eq:ene:d1}, we obtain
	\begin{equation*}
		\frac{\mathrm{d} P}{\mathrm{d}t}(t)=\frac{1}{4}\frac{\mathrm{d}}{\mathrm{d}t}\left(\frac{\left(\frac{\mathrm{d} y}{\mathrm{d}t}(t)\right)^2}{y(t)}\right)=\frac{\frac{\mathrm{d} y}{\mathrm{d}t}(t)}{4y^2(t)}\left(2y \frac{\mathrm{d}^2y}{\mathrm{d}t^2}(t)-\left(\frac{\mathrm{d} y}{\mathrm{d}t}(t)\right)^2\right).
	\end{equation*}
	Noticing
	\begin{equation*}
		2y \frac{\mathrm{d}^2y}{\mathrm{d}t^2}(t)-\left(\frac{\mathrm{d} y}{\mathrm{d}t}(t)\right)^2=4||(u,v)||_{L^2}^2||(u_t,v_t)||_{L^2}^2-4||(u,v)||_{L^2}^2K(t)-4\big\langle(u(t),v(t)),(u_t(t),v_t(t))\big\rangle^2_{(L^2\times L^2)^2},
	\end{equation*}
	we use the assumption $K(t)<0$ and apply Cauchy-Schwarz inequality to conclude that
	\begin{equation*}
		\frac{\mathrm{d} P}{\mathrm{d}t}(t)>0,
	\end{equation*}
	which is just the second part of the stated results in \textit{(i)}. Consequently $P(t)$ is increasing in $t\in[0,T]$.
	
	It remains to show the last item \textit{(iv)}.  On the one hand, by fundamental theorem of calculus, we have for any $t\in [0,T]$
	\begin{equation}\label{eq:306}
		y(0)-y(t)=-\int_0^t\frac{\mathrm{d}y}{\mathrm{d}\tau}(\tau)\mathrm{d}\tau\leq -t\frac{\mathrm{d} y}{\mathrm{d}t}(0),
	\end{equation}
	where we have used the non-negativeness and the monotonicity of $y$ on $[0,T]$. On the other hand, by first substituting
	\eqref{eq:ene:d1} into the defining expression \eqref{eq:proj:int} of $P$ and then evaluating at $t=0$ in the resulted expression, we get
	\begin{equation}\label{eq:305}
		2P(0)= \frac{1}{2}\frac{\left(\frac{\mathrm{d}y}{\mathrm{d}t}(0)\right)^2}{y(0)}.
	\end{equation}
	Summing \eqref{eq:306} and \eqref{eq:305}, we get
	\begin{equation*}
			y(0)-y(t)+2P(0)\leq \frac{1}{2}\frac{\left(\frac{\mathrm{d} y}{\mathrm{d}t}(0)\right)^2}{y(0)}-t\frac{\mathrm{d} y}{\mathrm{d}t}(0).
	\end{equation*}
	From this, we see that for $t\geq t_b:=\frac{1}{2}\frac{\frac{d y}{dt}(0)}{y(0)}$, the value of the term before the sign `$\leq$' is negative, which is just the desired result in \textit{(iv)}. This finishes the proof of Lemma \ref{lem:skg}.
\end{proof}

We next affirm that the assumption $K(t)<0$ in Lemma \ref{lem:skg} is indeed true.
\begin{lemma}\label{lem:skg:1}
	Under the assumption of \textbf{Theorem \ref{thm:skg:int}}, $K(t)$ satisfies
	\begin{enumerate}[(i)]
		\item $K(t)<0$ for $t\in[0,T_{max})$;
		\item if in addition $t_b=\frac{\langle(u_0,v_0),(u_1,v_1)\rangle_{(L^2\times L^2)^2}}{||u_0||_{L^2}^2+||v_0||_{L^2}^2}<T_{max}$, then for any $t\in [t_b,T_{max})$, the value of $y(0)-y(t)+2P(0)$ is always negative and
		\begin{equation*}
			K(t)\leq -2\|(u_t,v_t)\|_{L^2\times L^2}^2.
		\end{equation*}
	\end{enumerate}
\end{lemma}

\begin{proof}
	We first show $K(0)<0$. By definitions of $K(t)$ and $E(t)$, we observe
	\begin{eqnarray*}
		\frac{1}{4}K(0)=E(0)-\frac{1}{2}\frac{\big\langle(u_0,v_0),(u_1,v_1)\big\rangle^2_{(L^2\times L^2)^2}}{\|(u_0,v_0)\|_{L^2\times L^2}^2}-\frac{1}{2}\left\|(u_1,v_1)-\frac{\big\langle(u_0,v_0),(u_1,v_1)\big\rangle^2_{(L^2\times L^2)^2}}{\|(u_0,v_0)\|_{L^2\times L^2}^2}(u_0,v_0)\right\|_{L^2\times L^2}^2
	\end{eqnarray*}
\begin{equation*}
	\ \ \ \ \ \ \ \ \ \ -\frac{1}{4}\left(\|\nabla (u_0,v_0)\|_{L^2\times L^2}^2+\|(u_0,v_0)\|_{L^2\times L^2}^2\right).
\end{equation*}
	It then follows from the third assumption in \textbf{Theorem \ref{thm:skg:int}} that
	\begin{equation*}
		\frac{1}{4}K(0)\leq-\frac{1}{2}\left\|(u_1,v_1)-\frac{\big\langle(u_0,v_0),(u_1,v_1)\big\rangle^2_{(L^2\times L^2)^2}}{\|(u_0,v_0)\|_{L^2\times L^2}^2}(u_0,v_0)\right\|_{L^2\times L^2}^2-\frac{1}{4}\|\nabla (u_0,v_0)\|_{L^2\times L^2}^2.
	\end{equation*}
	This implies, together with the first assumption $\| (u_0,v_0)\|_{L^2\times L^2}^2>0$ and that
		{$\|\nabla (u_0,v_0)\|_{L^2\times L^2}^2>0$}, that $K(0)<0$.
	
	We next show $K(t)<0$ for {$t\in[0,T_{max})$}. {It follows from Proposition \ref{prop:cauchy} that $K(t)$ is continuous in the time variable $t$}.  We argue by contradiction, assuming that there exists $t_0$, so that $K(t_0)=0$ and $K(t)<0$ for $t\in [0,t_0)$. Observe that $t_0>0$, thanks to the continuity of $K(t)$ and $K(0)<0$. This allows us to apply Lemma \ref{lem:skg} on $[0,t_0]$ to obtain
	\begin{equation*}
		\left\|(u(t),v(t))\right\|_{L^2\times L^2} > 0, \forall t\in [0,t_0).
	\end{equation*}
	Thus $P(t)$ is well-defined for any $ t\in [0,t_0)$. By invoking to the defintions of $K(t)$ and $E(t)$, we have
	\begin{equation}\label{eq:equa:1}
		\frac{1}{4}K(t_0)=E(t_0)-\frac{1}{2}\frac{\big\langle(u(t_0),v(t_0)),(u_t(t_0),v_t(t_0))\big)\rangle^2_{(L^2\times L^2)^2}}{\|(u(t_0),v(t_0))\|_{L^2\times L^2}^2}-\frac{1}{4}\left(\|(u(t_0),v(t_0))\|_{L^2\times L^2}^2+\|\nabla (u(t_0),v(t_0))\|_{L^2\times L^2}^2\right)
	\end{equation}
	\begin{equation*}
		\ \ \ -		\frac{1}{2}\left\|(u_t(t_0),v_t(t_0))-\frac{\big\langle(u(t_0),v(t_0)),(u_t(t_0),v_t(t_0))\big\rangle_{(L^2\times L^2)^2}}{\|(u(t_0),v(t_0))\|_{L^2\times L^2}^2}(u(t_0),v(t_0))\right\|_{L^2\times L^2}^2.
	\end{equation*}
	This, together with the energy conservation law and the third assumption in Theorem \ref{thm:skg:int}, infers that
	\begin{equation}\label{eq:equa:2}
		\frac{1}{4}K(t_0)\leq \frac{1}{2}\frac{\big\langle(u_0,v_0),(u_1,v_1)\big\rangle^2_{(L^2\times L^2)^2}}{\|(u_0,v_0)\|_{L^2\times L^2}^2}+\frac{1}{4}\|(u_0,v_0)\|_{L^2\times L^2}^2-\frac{1}{2}\frac{\big\langle(u(t_0),v(t_0)),(u_t(t_0),v_t(t_0))\big\rangle^2_{(L^2\times L^2)^2}}{\|(u(t_0),v(t_0))\|_{L^2\times L^2}^2}
	\end{equation}
	\begin{equation*}
		-\frac{1}{4}\|(u(t_0),v(t_0))\|_{L^2\times L^2}^2
		=\frac{1}{4}\big(y(0)-y(t_0)\big)+\frac{1}{2}\big(P(0)-P(t_0)\big),
	\end{equation*}
	{which is strictly negative since $y(t)$ and $P(t)$ are strictly increasing in $t\in[0,t_0]$ by \textit{(ii)} in Lemma \ref{lem:skg}}. This contradicts the assumption that $K(t_0)=0$. So we have that  $K(t)<0$ for any $t\in[0,T_{max})$.
	
	It remains to show \textit{(ii)}. It follows from \eqref{eq:equa:1} and the third assumption in \textbf{Theorem \ref{thm:skg:int}} that for any $t\in [t_b,T_{max})$
	\begin{equation*}
		K(t)\leq 2\frac{\big\langle(u_0,v_0),(u_1,v_1)\big\rangle^2_{(L^2\times L^2)^2}}{\|(u_0,v_0)\|_{L^2\times L^2}^2}+\|(u_0,v_0)\|_{L^2\times L^2}^2-2\|(u_t(t),v_t(t))\|_{L^2\times L^2}^2-\|(u(t),v(t))\|_{L^2\times L^2}^2
	\end{equation*}
	\begin{equation*}
		= 2P(0)+y(0)-y(t)-2\|(u_t(t),v_t(t))\|_{L^2\times L^2}^2.
	\end{equation*}
	This is
	\begin{equation*}
		\leq -2\|(u_t(t),v_t(t))\|_{L^2\times L^2}^2
	\end{equation*}
	by \textit{(iv)} in Lemma \ref{lem:skg}. This completes the proof of Lemma \ref{lem:skg:1}.
\end{proof}

We are now ready to give
\begin{proof}[Proof of {$\mathbf{Theorem}$ \ref{thm:skg:int}}] 	We argue by contradiction, assuming $T_{max}=\infty$. By Lemma \ref{lem:skg:1}, we have $K(t)<0$ for any $t\in[0,\infty)$. This in turn allows us to use Lemma \ref{lem:skg} to conclude that
	\begin{equation*}
		K(t)\leq -2\|(u_t,v_t)\|_{L^2\times L^2}^2,\  \forall t\in [t_b,\infty).
	\end{equation*}
	Substituting this into the differential expression \eqref{eq:ene:d2}, it follows that for $t\in [t_b,\infty)$
	\begin{equation*}
		\frac{\mathrm{d}^2}{\mathrm{d}t^2}y(t)\geq 2\|(u_t,v_t)\|_{L^2\times L^2}^2+4\|(u_t,v_t)\|_{L^2\times L^2}^2=6\|(u_t,v_t)\|_{L^2\times L^2}^2.
	\end{equation*}
	By multiplying both sides by $y(t)$ and using Cauchy-Schwarz inequality, we see that $y(t)$ satisfies the differential inequality
	\begin{equation*}	y(t)\frac{\mathrm{d}^2}{\mathrm{d}t^2}y(t)-\frac{3}{2}\left(\frac{\mathrm{d}}{\mathrm{d}t}y(t)\right)^2 \geq 6\left(\|(u,v)\|_{L^2\times L^2}^2\|(u_t,v_t)\|_{L^2\times L^2}^2-\langle(u,u_t),(v,v_t)\rangle^2_{(L^2\times L^2)^2}\right)\geq 0,\ \ \forall t\in[t_b,\infty).
	\end{equation*}
	
	At time $t=t_b$, it follows \textbf{Lemma \ref{lem:skg}} that $y(t_b)>0$ and $\frac{\mathrm{d}}{\mathrm{d}t}y(t_b)>0$. We can then apply \textbf{Lemma \ref{lem:levine:1}} with
	\begin{equation*}
		\psi(t)=y(t), \gamma=\frac{3}{2}>1
	\end{equation*}
	to conclude that the solution $u(t,x)$ blows up in finite time.
\end{proof}

\section{Proof of Theorem \ref{thm:payne-sattinger:int}}\label{sect:proof:thm3}

In this section, we {are going to prove \textbf{Theorem} \ref{thm:payne-sattinger:int}}. For $(u_0,u_1),(v_0,v_1)\in \mathcal{H}$, \textbf{Proposition \ref{prop:cauchy}} implies that \eqref{eq:skg:int} admits a local in time solution $u(t)$ with the initial data $(u_0,u_1), (v_0,v_1)$.

We shall give some sufficient conditions on the initial data $(u_0,u_1), (v_0,v_1)$ so that the local solution can not be extended to a global one. The criterion we give here is in the spirit of Payne and Sattinger, with the help of stationary solutions to \eqref{eq:skg:int}.

\subsection{Existence of nontrivial stationary solutions}
Recall if  $(\phi,\psi)$ is a nontrivial stationary solution of \eqref{eq:skg:int}, then $(\phi,\psi)\in H^1\times H^1$ satisfies
\begin{equation}\label{eq:SKG:sta}
	\left\{
	\begin{split}
		-\Delta \phi+\phi&=\phi^3+\beta\psi^2\phi,\\
		-\Delta \psi+\psi&=\psi^3+\beta\phi^2\psi.
	\end{split}
	\right.
\end{equation}
This is indeed the Euler-Lagrangian equation of the functional
\begin{equation}\label{eq:2601}
	J[\phi,\psi]:=\frac{1}{2}\int_{\mathbb{R}^3}\left[|\nabla\phi|^2+|\phi|^2+|\nabla\psi|^2+|\psi|^2\right]\mathrm{d}x
 -\frac{1}{4}\int_{\mathbb{R}^3}\left[|\phi|^4+|\psi|^4+2\beta\phi^2\psi^2\right]\mathrm{d}x.
\end{equation}

{
For $(\phi,\psi)\in H^1\times H^1$, we let
\begin{equation}\label{eq:2602}
	j_{(\phi,\psi)}(\lambda):=J[\lambda\phi,\lambda\psi].
\end{equation}
for $\lambda>0$. About this functional, we have
\begin{lemma}\label{lem:mount:pass}
	If $(\phi,\psi)\in H^1\times H^1$ is not the zero element and $\beta> -1$, then
	\begin{enumerate}[(i)]
		\item there exists a unique number $\lambda_0>0$ so that $j_{(\phi,\psi)}(\lambda)>0$ for $\lambda\in(0,\lambda_0)$, while $j_{(\phi,\psi)}(\lambda)<0$ for $\lambda\in(\lambda_0,\infty)$.
		\item furthermore, we have
		$\lim_{\lambda\rightarrow0+}j_{(\phi,\psi)}(\lambda)=0$, and $\lim_{\lambda\rightarrow\infty}j_{(\phi,\psi)}(\lambda)=-\infty$.
	\end{enumerate}
\end{lemma}
\begin{remark}
	This result is indeed the analogue of the scalar version \cite[Lemma 2.8]{nakanishischlag2011}, see also \cite{nakanishischlag2011jde}.
\end{remark}
\begin{proof}
	We begin the proof by claiming
	\begin{equation}\label{eq:2603}
		\int_{\mathbb{R}^3}\left[|\phi|^4+|\psi|^4+2\beta\phi^2\psi^2\right]\mathrm{d}x\neq 0.
	\end{equation}
	Indeed, if this was not the case, we would have $\int_{\mathbb{R}^3}\phi^4\mathrm{d}x=\int_{\mathbb{R}^3}\psi^4\mathrm{d}x=0$, thanks to the assumption $\beta>-1$. We are first going to see what happens to $\phi$. Since $\phi\in L^2$, for any $\delta>0$, there exists $R>0$ such that $\|\phi\|_{L^2(|\cdot|\geq R)}<\delta$. On the bounded domain $\{x\in\mathbb{R}^3:|x|<R\}$, we use H\"{o}lder inequality to obtain
	\begin{equation}
		\|\phi\|_{L^2(|\cdot|< R)}\leq \left|\{x\in\mathbb{R}^3:|x|<R\}\right|^{1/2}\|\phi\|_{L^4(|\cdot|<R)},
	\end{equation}
	which is $0$, thanks to the assumption $\int_{\mathbb{R}^3}\phi^4\mathrm{d}x=0$. By combining these two results, we see that $\|\phi\|_{L^2(\mathbb{R}^3)}<\delta$. Since $\delta$ can be taken to be arbitrarily small, $\|\phi\|_{L^2(\mathbb{R}^3)}=0$. To proceed, we shall show $\|\nabla\phi\|_{L^2}=0$. For this we invoke the notion of Fourier transform. Let $\hat{\phi}(\xi)$ be the Fourier transform of $\phi$, then $\|\hat{\phi}\|^2_{L^2}=\|\phi\|^2_{L^2}=0$ and $ \|\phi\|^2_{H^1}=\|\xi\hat{\phi}\|^2_{L^2}+\|\hat{\phi}\|^2_{L^2}$, thanks to our assumption that $\phi\in H^1$. For $M$ a parameter to be specified later, we decompose $$\|\xi\hat{\phi}\|^2_{L^2}=\|\xi\hat{\phi}\|^2_{L^2(\{|\xi|\leq M\})}+\|\xi\hat{\phi}\|^2_{L^2(\{|\xi|> M\})}.$$
	Notice that $\|\xi\hat{\phi}\|^2_{L^2(\{|\xi|\leq M\})}\leq M^2\|\hat{\phi}\|^2_{L^2}=0, $ thus $ \|\nabla \phi\|^2_{L^2}=\|\xi\hat{\phi}\|^2_{L^2(\{|\xi|> M\})}$. Since $\phi\in H^1$,  for any $\delta>0$, we have $\|\xi\hat{\phi}\|^2_{L^2(\{|\xi|> M\})}\leq \delta$ for $M$ large enough. Thus $\|\nabla\phi\|^2_{L^2}=0$. Therefore, we get $\|\phi\|_{H^1}=0$. Applying the same argument to $\psi$, we can obtain $\|\psi\|_{H^1}=0$. Combining this with $\|\phi\|_{H^1}=0$, we see that $(\phi,\psi)$ is a zero element in $H^1\times H^1$. But this contradicts to the assumption. Therefore we have \eqref{eq:2603}.
	
	By the defining formulas \eqref{eq:2601} and \eqref{eq:2602}, we have
	\begin{equation}\label{eq:2604}
		j_{(\phi,\psi)}(\lambda)=\frac{\lambda^2}{2}\left(\|\phi\|^2_{H^1}+\|\psi\|^2_{H^1}\right)-\frac{\lambda^4}{4}\int_{\mathbb{R}^3}\left[|\phi|^4+|\psi|^4+2\beta\phi^2\psi^2\right]\mathrm{d}x.
	\end{equation}
	Since $\int_{\mathbb{R}^3}\left[|\phi|^4+|\psi|^4+2\beta\phi^2\psi^2\right]\mathrm{d}x\neq 0$ (as in \eqref{eq:2603}), we take $\lambda_0$ to be
	\begin{equation*}
		\lambda_0:=\sqrt{\frac{2\left(\|\phi\|^2_{H^1}+\|\psi\|^2_{H^1}\right)}{\int_{\mathbb{R}^3}\left[|\phi|^4+|\psi|^4+2\beta\phi^2\psi^2\right]\mathrm{d}x}}.
	\end{equation*}
Then $\lambda_0>0$. For $\lambda\in(0,\lambda_0)$, we have $j_{(\phi,\psi)}(\lambda)>0$. But for $\lambda\in(\lambda_0,\infty)$, we can check $j_{(\phi,\psi)}(\lambda)<0$. This finishes the proof of $(i)$.

It remains to show $(ii)$.  Since the two terms of $j_{(\phi,\psi)}(\lambda)$ have a common factor $\lambda^2$, we get $	\lim_{\lambda\rightarrow0+}j_{(\phi,\psi)}(\lambda)=0$. Notice that $j_{(\phi,\psi)}(\lambda)$ is indeed a polynomial in $\lambda$. By \eqref{eq:2603}, the coefficient of the highest order term is strictly negative. Thus as $\lambda$ tends to $+\infty$, $j_{(\phi,\psi)}(\lambda)$ tends to $-\infty$. This finishes the proof of $(ii)$ and completes the proof of Lemma \ref{lem:mount:pass}.
\end{proof}
}

By Lemma \ref{lem:mount:pass}, $j_{(\phi,\psi)}(\lambda)$ attains its maximum at some finite $\lambda$. This suggests us to renormalize $(\phi,\psi)$ so that at $\lambda=\lambda_{\ast}:=1$ we have
\begin{equation*}
	\partial_\lambda j_{(\phi,\psi)}(\lambda)\big|_{\lambda=\lambda_{\ast}}=0.
\end{equation*}
The left hand side is indeed the expression $K[\phi,\psi]$ defined by\eqref{eq:3000}.
 {Furthermore, the functional $K$ enjoys as well the mountain-pass geometry (see \cite{struwe2000} for detailed notions).}
 \begin{lemma}\label{lem:K:mountainpass}
 	If $(\phi,\psi)\in H^1\times H^1$ is not the zero element and $\beta> -1$, then
 	\begin{enumerate}[(i)]
 		\item there exists a unique number $\lambda_0>0$ so that $K[\lambda\phi,\lambda\psi]>0$ for $\lambda\in(0,\lambda_0)$, while $K[\lambda\phi,\lambda\psi]<0$ for $\lambda\in(\lambda_0,\infty)$;
 		\item  furthermore, we have $\lim_{\lambda\rightarrow0+}K[\lambda\phi,\lambda\psi]=0$ and $\lim_{\lambda\rightarrow\infty}K[\lambda\phi,\lambda\psi]=-\infty$.
 	\end{enumerate}
 \end{lemma}
This lemma can be proved by the same argument as that in Lemma \ref{lem:mount:pass}. So we omit details. We emphasize here that Lemma \ref{lem:K:mountainpass} will play an important role in the proof of Lemma \ref{lem:keyvariation}.

 We are in a position to find {the lowest mountain pass level} for \eqref{eq:SKG:sta}. If $\beta\geq 0$, then we apply \cite[\textbf{Proposition 3.5}]{sirakov07} with $\lambda_1= \mu_1=\mu_2=1$ to conclude that \eqref{eq:SKG:sta} has a nontrivial solution $\left(\tilde{\phi},\tilde{\psi}\right)$, each component of which is non-negative, spherically symmetric and decreasing. Since $\left(\tilde{\phi},\tilde{\psi}\right)$ is nonzero and $\beta$ is assumed to be non-negative,
the quantity, known as the lowest mountain pass level,
\begin{equation*}
	d:=\min\left\{\left. j_{(\phi,\psi)}(1)\ \right|\ K[\phi,\psi]=0,(\phi,\psi)\neq(0,0)\right\}=j_{(\tilde{\phi},\tilde{\psi})}(1)
\end{equation*}
equaling to
\begin{equation*}
	\frac{1}{4}\int_{\mathbb{R}^3}\left[|\tilde \phi|^4+|\tilde \psi|^4+2\beta\tilde{\phi}^2\tilde{\psi}^2\right] \mathrm{d}x,
\end{equation*}
{is strictly positive}.

\begin{remark}
	Indeed, in \cite{sirakov07}, Sirakov showed the existence of nontrivial solutions for a larger family of systems containing \eqref{eq:SKG:sta}. Nevertheless, several other authors generalized this result to more general systems, see \cite{pengwangwang2019}, \cite{weiwu2020},\cite{tavares2011},\cite{soavetavares2016}, \cite{chenzou2012} and \cite{chenzou2015}. What's more, some authors even obtained the uniqueness of these solutions, see for instance \cite{ikoma2009}, \cite{anchernshi2018},\cite{linwei} and \cite{linwei08}.
\end{remark}


	In order to use $d$ to give some criterion of the blowing up for solution of \eqref{eq:skg:int}, we need the characterization of $d$ from the view point of the constrained minimization (see \cite{nakanishischlag2011} and \cite{nakanishischlag2011jde} for the scalar case). More precisely, we define the functional
	\begin{equation*}
		E^1[u_1,u_2]:= \frac{1}{4}\left(\|u_1\|_{H^1}^2+\|u_2\|_{H^1}^2\right)
	\end{equation*}
	for $\left(u_1,u_2\right)\in H^1\times H^1$. Define the constrained set to be
	\begin{equation}\label{Np}
		\mathcal{N}':=\left\{\left.(u_1,u_2)\in H^1\times H^1\right|,
		(u_1,u_2)\not= (0,0),
		K[u_1,u_2]\leq 0   \right\}.
	\end{equation}
	We will consider the attainability of the quantity
	\begin{equation}\label{vp:n}
		d_1:=\inf_{(u_1,u_2)\in \mathcal{N}'} E^1[u_1,u_2].
	\end{equation}

	Indeed, this is affirmed in the following result.
	\begin{lemma}\label{lem:keyvariation}	Fix $\beta\geq 0$. Then $d_1$
		is attained by some element in $\mathcal{N}'$. Moreover, $d_1=d$.
	\end{lemma}
	Note that \textbf{Theorem \ref{thm:char:d:int}} follows directly from this lemma.
	\begin{proof} We follow the lines in the proof of {Proposition 3.5 in \cite{sirakov07}}.
		Let $\{(v_{m,1},v_{m,2})\}\subset \mathcal{N}' $ be a minimizing sequence for $d_1$, then we have
		\begin{equation}\label{eq:2506}
			\|v_{m,1}\|_{H^1}^2+\|v_{m,2}\|_{H^1}^2\rightarrow_{m\rightarrow \infty} 4d_1.
		\end{equation}
		It follows that the sequence $\{(v_{m,1},v_{m,2})\}$ is bounded in $H^1\times H^1$.
		
		In order to capture the compactness result from this boundedness, we invoke to the Schwarz symmetrization. For each $m$, let $v_{m,i}^*$ be the Schwarz spherical rearrangement of $v_{m,i},i=1,2$. On the one hand, for each $m$ and each $i\in\{1,2\}$, $v_{m,i}^*$ is a radial function. On the other hand, by Lemma \ref{lem:rearrangement}, we have
		\begin{equation}
			E^1[v_{m,1}^*,v_{m,2}^*]\leq E^1[v_{m,1},v_{m,2}].
		\end{equation}
		Combining these two points with \eqref{eq:2506}, we see that $\{(v_{m,1}^*,v_{m,2}^*)\}$ is a bounded sequence in $H_r^1\times H_r^1$. Here $H^1_r$ is the subspace of $H^1$, consisting of functions that are spherically symmetric. Thus, by \cite[$\mathbf{Theorem\ A.I}'$]{BL83}, we conclude that there exists a couple  $(v_{*,1},v_{*,2})\in H^1_r\times H^1_r$ such that $(v_{m,1}^*,v_{m,2}^*)$ converges weakly $(v_{*,1},v_{*,2})$ in $H^1_r\times H^1_r$ and strongly to $(v_{*,1},v_{*,2})$ in $L^4\times L^4$ (up to subsequence. Here we denoted this subsequence still by $\{(v_{m,1}^*,v_{m,2}^*)\}$).

			We shall show $(v_{*,1},v_{*,2})\in \mathcal{N}'$.
			Applying lower semi-continuity and then using Lemma \ref{lem:rearrangement}, we first obtain
			\begin{equation}\label{eq:2507}
				\|(v_{*,1},v_{*,2})\|^2_{H^1_r\times H^1_r}\leq {\liminf}_{m\rightarrow \infty} \|(v_{m,1}^*,v_{m,2}^*)\|^2_{H^1_r\times H^1_r}\leq{\liminf}_{m\rightarrow \infty} \|(v_{m,1},v_{m,2})\|^2_{H^1\times H^1}.
			\end{equation}
			Since for each $m$, $(v_{m,1},v_{m,2})\in\mathcal{N}'$, we have $K[v_{m,1},v_{m,2}]\leq 0$ and hence
			\begin{equation*}
				\|(v_{m,1},v_{m,2})\|^2_{H^1\times H^1} \leq \int_{\mathbb{R}^3} \left[v_{m,1}^4+v_{m,2}^4+2\beta v_{m,1}^2v_{m,2}^2\right]\mathrm{d}x,\ \ \forall m.
			\end{equation*}
			Combining this with \eqref{eq:2507}, we get
			\begin{equation}\label{eq:2508}
				\|(v_{*,1},v_{*,2})\|^2_{H^1_r\times H^1_r}\leq \liminf_{m\rightarrow\infty}\int_{\mathbb{R}^3} \left[v_{m,1}^4+v_{m,2}^4+2\beta v_{m,1}^2v_{m,2}^2\right]\mathrm{d}x.
			\end{equation}
			We next bound the right hand side of \eqref{eq:2508}. For each $m$, {by applying Lemma \ref{lem:rearrangement}	with $p=4$ and using our assumption $\beta\geq 0$, we have
			\begin{equation}\label{eq:2509}
				\int_{\mathbb{R}^3} \left[v_{m,1}^4+v_{m,2}^4+2\beta v_{m,1}^2v_{m,2}^2\right]dx\leq \int_{\mathbb{R}^3} \left[\left(v_{m,1}^*\right)^4+\left(v_{m,2}^*\right)^4+2\beta\left(v_{m,1}^*\right)^2\left(v_{m,2}^*\right)^2\right]\mathrm{d}x.
			\end{equation}}
			By the strong convergence of $\{(v_{m,1}^*,v_{m,2}^*)\}$, we see
			\begin{equation}\label{eq:2515}
				\liminf_{m\rightarrow \infty} \int_{\mathbb{R}^3}\left[\left(v_{m,1}^*\right)^4+\left(v_{m,2}^*\right)^4\right]\mathrm{d}x= \int_{\mathbb{R}^3}\left[\left(v_{*,1}\right)^4+\left(v_{*,2}\right)^4\right]\mathrm{d}x.
			\end{equation}
			In order to get the convergence of the remaining part of the right hand side of \eqref{eq:2509}, we do the following calculations
			\begin{equation}\label{eq:2511}
				\int_{\mathbb{R}^3}\left|\left(v_{m,1}^*\right)^2\left(v_{m,2}^*\right)^2-\left(v_{*,1}\right)^2\left(v_{*,2}\right)^2\right|\mathrm{d}x\leq \int_{\mathbb{R}^3}\left|\left(v_{m,1}^*\right)^2\left(v_{m,2}^*\right)^2-\left(v_{m,1}^*\right)^2\left(v_{*,2}\right)^2\right|\mathrm{d}x+
			\end{equation}
			\begin{equation*}
				\ \ \ \ \ \ \ \ \ \ \ \ \ \ \ \ \ \ \ \ \ \ \ \ \ \ \ \ \ \ \ \ \ \ \ \ \ \ \ \ \ \ \ \ \ \ \ \ \ \ \ \ \ \ \ \ \ \ \ \ \ \ 	+\int_{\mathbb{R}^3}\left|\left(v_{m,1}^*\right)^2\left(v_{*,2}\right)^2-\left(v_{*,1}\right)^2\left(v_{*,2}\right)^2\right|\mathrm{d}x=:I+II.
			\end{equation*}
			Using H\"{o}lder's inequality, we can bound $I$ as follows
			\begin{equation}
				I\leq \left\|v_{m,1}^*\right\|^2_{L^4}\left\|v_{m,2}^*+v_{*,2}\right\|_{L^4}\times \left\|v_{m,2}^*-v_{*,2}\right\|_{L^4},
			\end{equation}
			which tends to $0$ as $m$ tends to infinity, thanks to the strong convergence of $\{(v_{m,1}^*,v_{m,2}^*)\}$ to $(v_{*,1},v_{*,2})$ in $L^4\times L^4$. Hence
			\begin{equation}\label{eq:2512}
				\liminf_{m\rightarrow\infty}\int_{\mathbb{R}^3}\left|\left(v_{m,1}^*\right)^2\left(v_{m,2}^*\right)^2-\left(v_{m,1}^*\right)^2\left(v_{*,2}\right)^2\right|\mathrm{d}x=0.
			\end{equation} Applying a similar argument to $II$, we can see that
			\begin{equation}\label{eq:2513}
				\liminf_{m\rightarrow\infty}\int_{\mathbb{R}^3}\left|\left(v_{m,1}^*\right)^2\left(v_{*,2}\right)^2-\left(v_{*,1}\right)^2\left(v_{*,2}\right)^2\right|\mathrm{d}x=0.
			\end{equation}
			By sending $m$ to infinity on both sides of \eqref{eq:2511} and using \eqref{eq:2512} and \eqref{eq:2513}, we obtain that $\left(v_{m,1}^*\right)^2\left(v_{m,2}^*\right)^2$ converges strongly to $\left(v_{*,1}\right)^2\left(v_{*,2}\right)^2$ in $L^1$. Consequently we get
			\begin{equation}\label{eq:2514}
				\liminf_{m\rightarrow\infty}\int_{\mathbb{R}^3} \left[2\beta v_{m,1}^2v_{m,2}^2\right]\mathrm{d}x=\int_{\mathbb{R}^3}2\beta\left(v_{*,1}\right)^2\left(v_{*,2}\right)^2\mathrm{d}x.
			\end{equation}
			By sending $m$ to infinity on both sides of \eqref{eq:2509}, and using \eqref{eq:2514} and \eqref{eq:2515}, we see
			\begin{equation}\label{eq:2516}
				\liminf_{m\rightarrow\infty}\int_{\mathbb{R}^3} \left[v_{m,1}^4+v_{m,2}^4+2\beta v_{m,1}^2v_{m,2}^2\right]\mathrm{d}x=\int_{\mathbb{R}^3}\left[\left(v_{*,1}\right)^4+\left(v_{*,2}\right)^4+2\beta\left(v_{*,1}\right)^2\left(v_{*,2}\right)^2\right]\mathrm{d}x.
			\end{equation}
			Substituting this equality into \eqref{eq:2508}, we get
			\begin{equation}
				\|(v_{*,1},v_{*,2})\|^2_{H^1_r\times H^1_r}\leq \int_{\mathbb{R}^3}\left[\left(v_{*,1}\right)^4+\left(v_{*,2}\right)^4+2\beta\left(v_{*,1}\right)^2\left(v_{*,2}\right)^2\right]\mathrm{d}x.
			\end{equation}
			From this, it follows that {$K[v_{*,1},v_{*,2}]\leq 0$}.
			
			It remains to prove $(v_{*,1},v_{*,2})\neq(0,0)$ in $H^1\times H^1$. Since $\{\left(v_{m,1},v_{m,2}\right)\}\subset\mathcal{N}'$, we have for each $m$
			\begin{equation}\label{eq:2519}
				\|(v_{m,1},v_{m,2})\|^2_{H^1\times H^1}\leq \int_{\mathbb{R}^3}\left[v_{m,1}^4+v_{m_2}^4+2\beta v_{m,1}^2v_{m,2}^2\right]\mathrm{d}x.
			\end{equation}
			We bound further this inequality as follows
			\begin{equation}\label{eq:2520}
				\int_{\mathbb{R}^3}\left[v_{m,1}^4+v_{m_2}^4+2\beta v_{m,1}^2v_{m,2}^2\right]\mathrm{d}x\leq {(1+\beta)}\int_{\mathbb{R}^3}\left[v_{m,1}^4+v_{m_2}^4\right]\mathrm{d}x
			\end{equation}
			\begin{equation*}
				\ \ \ \ \ \ \ \ \ \ \ \ \ \ \ \ \ \ \ \ \ \ \ \ \ \ \ \ \ \ \ \ \ \ \ \ \ \ \ \ \ \ \ \ \ \ \ \ \ \ \ \ \ \ \ \ \ \ \ \
				\leq C\left[\int_{\mathbb{R}^3}\left[v_{m,1}^4+v_{m_2}^4\right]\mathrm{d}x\right]^{\frac{1}{2}}\times \|(v_{m,1},v_{m,2})\|^2_{H^1\times H^1}.
			\end{equation*}
			where we used H\"{o}lder's inequality and the elementary inequality $2ab\leq a^2+b^2$ in the first inequality, and the Sobolev embedding $H^1\subset L^4$ in the second inequality. Here the positive constant $C$ depends on the Sobolev embedding and $\beta$. Combining \eqref{eq:2519} and \eqref{eq:2520}, we get for each $m$
			\begin{equation}\label{eq:2521}
				\|(v_{m,1},v_{m,2})\|^2_{H^1\times H^1}\leq C\left[\int_{\mathbb{R}^3}\left[v_{m,1}^4+v_{m_2}^4\right]\mathrm{d}x\right]^{\frac{1}{2}}\times \|(v_{m,1},v_{m,2})\|^2_{H^1\times H^1}.
			\end{equation}
			It follows from $(v_{m,1},v_{m,2})\in \mathcal{N}'$, that $\|(v_{m,1},v_{m,2})\|^2_{H^1\times H^1}>0$. Thus, we can cancel out the common factor $\|(v_{m,1},v_{m,2})\|^2_{H^1\times H^1}$ on both sides of \eqref{eq:2521} to conclude
			\begin{equation}
				1<C\left[\int_{\mathbb{R}^3}\left[v_{m,1}^4+v_{m,2}^4\right]\mathrm{d}x\right]^{\frac{1}{2}}=C\left[\int_{\mathbb{R}^3}\left[\left(v_{m,1}^*\right)^4+\left(v_{m,2}^*\right)^4\right]\mathrm{d}x\right]^{\frac{1}{2}}
			\end{equation}
			where in the equality, we have used  Lemma \ref{lem:rearrangement} with $p=4$.
			By the strong convergence of $\{(v_{m,1}^*,v_{m,2}^*)\}$ in $L^4\times L^4$ once again, we can send $m$ to $\infty$ to obtain
			\begin{equation}\label{eq:2522}
				1\leq C\left[\int_{\mathbb{R}^3}\left[\left(v_{*,1}\right)^4+\left(v_{*,2}\right)^4\right]\mathrm{d}x\right]^{\frac{1}{2}}.
			\end{equation}
			We use the Sobolev embedding to bound the right hand side of \eqref{eq:2522} to obtain
			\begin{equation}
				1\leq \tilde{C}\|(v_{*,1},v_{*,2})\|^2_{H^1\times H^1}
			\end{equation}
			where $\tilde{C}$ is a positive constant, different from $C$. Consequently $(v_{*,1},v_{*,2})\neq (0,0)$ in $H^1\times H^1$.
			
			Combining these two results $K[v_{*,1},v_{*,2}]\leq 0$ and $(v_{*,1},v_{*,2})\neq (0,0)$ in $H^1\times H^1$, we see that $(v_{*,1},v_{*,2})\in \mathcal{N}'$. This finishes the proof of first part of Lemma \ref{lem:keyvariation}.
			
			We turn to the proof of second part. We note that if $(v_1,v_2)\in H^1\times H^1$ satisfies $K[v_1,v_2]=0$, then $E^1[v_1,v_2)]=J[v_1,v_2]$. Then by definitions of $J$ and $E^1$, we get
			\begin{equation}\label{eq:2523}
				d_1\leq d.
			\end{equation}
			
			To show the reversed inequality, we argue by contradiction by assuming $K[v_{*,1},v_{*,2}]<0$. Since $(v_{*,1},v_{*,2})\neq (0,0)$ in $H^1\times H^1$, it then follows from Lemma \ref{lem:K:mountainpass}, that there exists a constant $\lambda\in (0,1)$, such that $K[\lambda v_{*,1},\lambda v_{*,2}]=0$. Thus we have
			\begin{equation}
				d\leq J[\lambda v_{*,1},\lambda v_{*,2}]=\frac{1}{4}E^1[\lambda v_{*,1},\lambda v_{*,2}]= \frac{\lambda^2 }{4}E^1[v_{*,1},v_{*,2}]<\frac{1}{4}E^1[ v^*]=d_1.
			\end{equation}
			But this contradicts with the known inequality \eqref{eq:2523}. This contradiction implies $$K[v_{*,1},v_{*,2}]\geq0.$$ Combining this with the first part (that is,$(v_{*,1},v_{*,2})\in\mathcal{N}'$), we can only have
			\begin{equation*}
				K[v_{*,1},v_{*,2}]=0.
			\end{equation*}
			Consequently $d_1=\frac{1}{4}E^1[v_{*,1},v_{*,2}]=J[v_{*,1},v_{*,2}]\geq d$. This finishes the proof of second part of Lemma \ref{lem:keyvariation} and hence completes the proof of Lemma \ref{lem:keyvariation}.

	\end{proof}


\subsection{Blow up criterion}
As is alluded to in the introduction part, we used $d$ to define two subsets of initial datum: $\mathcal{W}$ and $\mathcal{E}$. In order to prove Theorem \ref{thm:payne-sattinger:int}, we first  affirm the fact that $\mathcal{E}$ is invariant under the flow associated to \eqref{eq:skg}, as is stated in the following.
\begin{proposition}\label{prop:inv:ext}
	$\mathcal{E}$ is invariant under the flow of \eqref{eq:skg:int}.
\end{proposition}
\begin{proof}
	Given $\left((u_0,u_1),(v_0,v_1)\right)\in\mathcal{E}$, let $\left(u(t),v(t)\right)$ be its corresponding solution to \eqref{eq:skg}. We shall show that as long as the time variable $t$ is in the lifespan, we will still have $\left((u(t),\partial_tu(t)),(v(t),\partial_tv(t))\right)$ will remain in $\mathcal{E}$.
	
	We argue by contradiction. Assume $\left((u(t),\partial_tu(t)),(v(t),\partial_tv(t))\right)$ leaves $\mathcal{E}$ at time $t_0$, we then have
	\begin{equation*}
		K\left[u(t_0),v(t_0)\right]=0.
	\end{equation*}
	Since $K[u(t),v(t)]$ {is continuous in $t$} and $K[u_0,v_0]<0$, we can assume $t_0>0$ is the first time at which $\left((u(t),\partial_tu(t)),(v(t),\partial_tv(t))\right)$ leaves $\mathcal{E}$. Then for any subsequence $\{t_n\}$ satisfying $t_n\rightarrow t_0$ from left as $n\rightarrow\infty$, we have
	\begin{equation*}
		K\left[u(t_0),v(t_0)\right]\leq \liminf_{n\rightarrow\infty}K\left[u(t_n),v(t_n)\right]\leq 0.
	\end{equation*}
	By the contradiction hypothesis, the equality holds here.
	\begin{claim}\label{claim:4}
		There exists such a subsequence $\{t_n\}$ such that $\left(u(t_n),v(t_n)\right)\neq (0,0)$ in $H^1\times H^1$.
	\end{claim}
	\begin{proof}[Proof of Claim \ref{claim:4}]
		Assume we do not have such a subsequence, then $\left(u(t),v(t)\right)$ vanishes (as element in $H^1\times H^1$) in a left open neighborhood of $t_0$, say \begin{equation}\label{eq:1200}
			\left(u(t),v(t)\right)=(0,0),\ \forall t\in(t_0-\delta,t_0)
		\end{equation}
	for some small $\delta>0$. By the Sobolev embedding $H^1\subset L^4$, there holds for such $t$
	\begin{equation}\label{eq:1201}
		\int_{\mathbb{R}^3} u^4\mathrm{d}x=\int_{\mathbb{R}^3}v^4\mathrm{d}x=0.
	\end{equation}
	Substituting \eqref{eq:1200} and \eqref{eq:1201} into the defining expressions of $J$ and $K$, we get
	\begin{equation}
		J\left[u(t),v(t)\right]=K\left[u(t),v(t)\right]=0,\ \forall t\in(t_0-\delta,t_0).
	\end{equation}
	{This implies that $K[u(t_0-\delta),v(t_0-\delta)]=0$ by the {continuity of $K[u(t),v(t)]$ in the time variable $t$}}. We know from this that $\left((u(t),\partial_tu(t)),(v(t),\partial_tv(t))\right)$ leaves $\mathcal{E}$ at time $t_0-\delta$, before $t_0$. But this contradicts our assumption on $t_0$. This ends the proof of Claim \ref{claim:4}.	
	\end{proof}
	Let's continue the proof of Proposition \ref{prop:inv:ext}. By Claim \ref{claim:4}, we can take the sequence $\{t_n\}$ so that $\left(u(t_n),v(t_n)\right)\neq (0,0)$ in $H^1\times H^1$. Then by \textbf{Corollary \ref{lem:tech}}, we get
	\begin{equation}
		\liminf_{n\rightarrow\infty}J[u(t_n),v(t_n)]\geq d.
	\end{equation}
	On the one hand, by the definition of $\mathcal{E}$, the RHS of this inequality is strictly bigger that $E(t_0)$. On the other hand, by definitions of $J$ and $E$, the LHS is smaller than $E(t_0)$. Combining these two points, we get $d>d$, which is impossible. This contradiction implies that $\left((u(t),\partial_tu(t)),(v(t),\partial_tv(t))\right)$ can not leave $\mathcal{E}$, which completes the proof of \textbf{Proposition \ref{prop:inv:ext}}.
\end{proof}

We are ready to give
\begin{proof}[Proof of Theorem \ref{thm:payne-sattinger:int}]
	We first prove \textit{(i)}. Let $\left((u_0,u_1),(v_0,v_1)\right)\in\mathcal{W}$ and $\left(u(t),v(t)\right)$  the corresponding local solution of \eqref{eq:skg}.
	\begin{claim} \label{claim:1}As long as the solution $\left(u(t),v(t)\right)$ exists at time $t$, we have
		$K[u(t),v(t)]\geq 0$.
	\end{claim}
	
	By Claim \ref{claim:1} and the definition of $K[u(t),v(t)]$, we obtain
	\begin{equation*}
		\int_{\mathbb{R}^3}\left[u^4+v^4+2\beta u^2v^2\right] \mathrm{d}x\leq \int_{\mathbb{R}^3}\left[|\nabla u|^2+u^2+|\nabla v|^2+v^2\right]\mathrm{d}x,
	\end{equation*}
	which in turn, together with the definition of the energy, implies
	\begin{equation*}
		E[u(t),v(t)]\geq \frac{1}{4}\left(\|u(t)\|^2_{H^1}+\|v(t)\|^2_{H^1} \right)+\frac{1}{2}\left(\|\partial_tu(t)\|^2_{L^2}+\|\partial_tv(t)\|^2_{L^2} \right).
	\end{equation*}
	Together with the energy conservation, this implies
	\begin{equation*}
		\left\|\left((u(t),\partial_tu(t)),(v(t),\partial_tv(t))\right)\right\|_{\mathcal{H}\times\mathcal{H}}\sim E[u_0,v_0]
	\end{equation*}
	as long as the solution $\left(u(t),v(t)\right)$ exists at the time $t$. Then by {\textbf{Proposition} \ref{prop:cauchy}}, we can extend the local solution onto a larger time interval. If we repeat the above argument on the extended time interval, we can extend the solution to be defined on the whole time line. Therefore, to finish the proof of \textit{(i)}, it remains to give
	\begin{proof}[Proof of  Claim \ref{claim:1}]
			Assume by contradiction, there was some time $t_0$ so that 	$K_0[u(t_0),v(t_0)]<0$. We note that $t_0\neq 0$, because we have by the assumption $\left((u_0,u_1),(v_0,v_1)\right)\in\mathcal{W}$ that $K_0[u_0,v_0]=0$. Since 	$K_0[u(t),v(t)]$ is continuous in $t$, there exists a maximal interval $(t_-,t_+)$ containing $t_0$, on which 	$K[u(t),v(t)]<0$. It is clear that $0$ is not in $(t_-,t_+)$, so that we can first assume $t_+$ is close to $0$. From the ordered structure of the real line, we get from this that $t_+<0$ and hence $t_0<0$. If we restrict the solution $\left(u(t),v(t)\right)$ of the equation \eqref{eq:skg} onto time interval $[t_0,0]$, we see that $\left((u(t),\partial_tu(t)),(v(t),\partial_tv(t))\right)$ will enter $\mathcal{W}$ at time $t=0$, before which it is in $\mathcal{E}$. This contradicts the invariance of $\mathcal{E}$ under the flow of \eqref{eq:skg}. Thus we are left with the case that $t_-$ is close to $0$. In this case, by once again the order structure of the real line, we get from this that $t_->0$ and hence $t_0>0$. If we restrict the solution $\left(u(t),v(t)\right)$ of the equation \eqref{eq:skg} onto  time interval $[0,t_0]$, we see that $\left((u(t),\partial_tu(t)),(v(t),\partial_tv(t))\right)$ will enter $\mathcal{W}$ at time $t=0$ in the reversed time direction, after which it is in $\mathcal{E}$. Since \eqref{eq:skg} is invariant if the time is reversed, this also contradicts the invariance of $\mathcal{E}$ under the flow of \eqref{eq:skg}. These two contradictions show that there is no such time $t_0$ at which $K_0[u(t),v(t)]<0$. Hence this completes the proof of Claim \ref{claim:1}.
	\end{proof}
	
	We next prove the second assertion \textit{(ii)}. Let $\left((u_0,u_1),(v_0,v_1)\right)\in\mathcal{E}$ and $\left(u(t),v(t)\right)$ its corresponding solution to \eqref{eq:skg:int}. We argue by contradiction. Assume the solution exists for all time.
	
 By Proposition \ref{prop:inv:ext}, we see that $\left((u(t),\partial_tu(t)),(v(t),\partial_tv(t))\right)$ remains in $\mathcal{E}$ for any $t\geq 0$. Consequently we have $K[u(t),v(t)]<0$. Substituting this strict inequality into \eqref{eq:ene:d2}, we get that
				\begin{equation}\label{eq:1300}
			\frac{\mathrm{d}^2}{\mathrm{d}t^2}y(t)=2\left(\int_{\mathbb{R}^3}|u_t|^2+|v_t|^2\mathrm{d}x\right)-K[u(t),v(t)]>0,
		\end{equation}
	for any $t\geq 0$. From this we know that $y(t)$ is a strictly convex function of $t$. What's more, we have the following strong result.
	
	\begin{claim}\label{claim:5}
		There exists a time $t_\ast$, such that for all $t\geq t_\ast$ we have
		\begin{equation}
			\frac{\mathrm{d}}{\mathrm{d}t}y(t)>0.
		\end{equation}
	\end{claim}

\begin{proof}[Proof of Claim \ref{claim:5}]
	We argue by contradiction. Assume $\frac{\mathrm{d}}{\mathrm{d}t}y(t)\leq 0$ for all $t\geq 0$. Under this assumption, if $y(0)=0$, then we have $y(t)=0$ for all $t\geq0$. Consequently $\frac{\mathrm{d}^2}{\mathrm{d}t^2}y(t)=\frac{\mathrm{d}}{\mathrm{d}t}y(t)=0$ for all $t\geq 0$. But this contradicts \eqref{eq:1300}. Thus in the following, we assume $y(0)>0$.
	
	Since $y(t)$ is convex and decreasing, we may assume that $y(t)$ converges to some non-negative number $A$.
	\begin{claim}\label{claim:6}
		$A >0.$
	\end{claim}
	Admitting this claim, we can take a subsequence $(t_n)$ with $t_n\rightarrow\infty$ as $n$ tends to infinity, such that
	\begin{equation*}
		y(t_n)\rightarrow A,\ \ \frac{\mathrm{d}}{\mathrm{d}t}y(t_n)\rightarrow 0,
	\end{equation*}
and \begin{equation}\label{eq:1400}
	\frac{\mathrm{d}^2}{\mathrm{d}t^2}y(t_n)\rightarrow0,
\end{equation}
as $n$ tends to infinity. Using $K[u(t),v(t)]<0$, we substitute \eqref{eq:1300} into this last limitation to obtain
\begin{equation}\label{eq:1401}
	\|\partial_tu(t_n)\|^2_{L^2(\mathbb{R}^3)}+\|\partial_tv(t_n)\|^2_{L^2(\mathbb{R}^3)}\rightarrow_{n\rightarrow\infty}0.
\end{equation}
We then have
\begin{equation}\label{eq:1302}
	\lim_{n\rightarrow\infty}J[u(t_n),v(t_n)]=\lim_{n\rightarrow{\infty}}E[u(t_n,v(t_n))]=E[u_0,v_0]
\end{equation}
where in the last equality we have used the conservation of energy. On the other hand, since $A>0$, we can take $n$ large enough, say $n\geq L$ for some $L>0$ so that $y(t_n)\neq 0$. Then by Lemma \ref{lem:tech}, we have \begin{equation}\label{eq:1301}
	\liminf_{n\rightarrow\infty}J[u(t_n),v(t_n)]\geq d.
\end{equation}
Recall that $\left((u_0,u_1),(v_0,v_1)\right)\in\mathcal{E}$ and hence $d>E(u_0,v_0)$. Using this strict inequality, we combine \eqref{eq:1302} and \eqref{eq:1301} to conclude
\begin{equation*}
	\liminf_{n\rightarrow\infty}J[u(t_n),v(t_n)]>\lim_{n\rightarrow\infty}J[u(t_n),v(t_n)],
\end{equation*}
which is impossible. Thus the initial assumption is false. This finishes the proof of Claim \ref{claim:5}, modulo the proof of Claim \ref{claim:6}.
\end{proof}

We now give
\begin{proof}[Proof of Claim \ref{claim:6}]
	We argue by contradiction. Assume that $A=0$. In this case, we still have the limitation \eqref{eq:1400} and hence \eqref{eq:1401} and \eqref{eq:1302} as well. Combining \eqref{eq:1400} with \eqref{eq:1300}, we get
	\begin{equation}\label{eq:1402}
		\lim_{n\rightarrow{\infty}} K[u(t_n),v(t_n)]=0.
	\end{equation}
	If we substitute the expressions of $J[u(t_n),v(t_n)]$ and $K[u(t_n),v(t_n)]$, and this limitation into \eqref{eq:1302}, we can get
	\begin{equation}\label{eq:1405}
		\lim_{{n\rightarrow\infty}}\int_{\mathbb{R}^3}\left[u^4(t_n)+v^4(t_n)+2\beta u^2(t_n)v^2(t_n)\right]\mathrm{d}x=4E[u_0,u_1;v_0,v_1]
	\end{equation}
	and
	\begin{equation}\label{eq:1406}
		\lim_{n\rightarrow\infty}\left[\|u(t_n)\|^2_{H^1}+\|v(t_n)\|^2_{H^1}\right]=4E[u_0,u_1;v_0,v_1].
	\end{equation}
	This implies that $E[u_0,u_1;v_0,v_1]\geq 0$. We are thus in the following two alternatives.
	\begin{itemize}
		\item[Case 1:] $E[u_0,u_1;v_0,v_1]> 0$. By using the interpolation and Sobolev inequality (or directly Gargliado-Nirenberg inequality), we get for each $n$
		\begin{equation}\label{eq:1403}
			\int_{\mathbb{R}^3}\left[u^4(t_n)+v^4(t_n)+2\beta u^2(t_n)v^2(t_n)\right]\mathrm{d}x\leq C \left(\|u(t_n)\|^2_{L^2}+\|v(t_n)\|^2_{L^2}\right)^{\frac{1}{2}} \left(\|u(t_n)\|^2_{H^1}+\|v(t_n)\|^2_{H^1}\right)^{\frac{3}{2}}.
		\end{equation}
		We first bound LHS of \eqref{eq:1403} from below. By \eqref{eq:1405}, we can take $N_1>0$ so that for all $n\geq N_1$ there holds
		\begin{equation}\label{eq:1407}
			\int_{\mathbb{R}^3}\left[u^4(t_n)+v^4(t_n)+2\beta u^2(t_n)v^2(t_n)\right]\mathrm{d}x>2E[u_0,u_1;v_0,v_1].
		\end{equation}
		We next bound RHS of \eqref{eq:1403} from above. By \eqref{eq:1406}, we can take $N_2>0$ so that for all $n\geq N_2$ we have
		\begin{equation}\label{eq:1408}
			\left[\|u(t_n)\|^2_{H^1}+\|v(t_n)\|^2_{H^1}\right]<6E[u_0,u_1;v_0,v_1].
		\end{equation}
		Since $A=0$, for the same $\delta>0$, we can take $N_3>0$ so that for all $n\geq N_3$ we have
		\begin{equation}\label{eq:1409}
			\|u(t_n)\|^2_{L^2}+\|v(t_n)\|^2_{L^2}<\delta.
		\end{equation}
		
		Inserting \eqref{eq:1407},\eqref{eq:1408} and \eqref{eq:1409} into \eqref{eq:1403}, we get for all $n\geq \max(N_1,N_2,N_3)$
		\begin{equation}
			2E[u_0,u_1;v_0,v_1]<C\delta^\frac{1}{2}\times(6E[u_0,u_1;v_0,v_1])^{\frac{3}{2}}.
		\end{equation}
		Since $E[u_0,u_1;v_0,v_1]>0$, if we take $\delta>0$ small enough, we can get \begin{equation*}
			E[u_0,u_1;v_0,v_1]>d.
		\end{equation*}
		This is imppossible, since the data is in $\mathcal{E}$.
		\item[Case 2:] $E[u_0,u_1;v_0,v_1]=0$. Since the data is in $\mathcal{E}$ and $\mathcal{E}$ is invariant under the flow of \eqref{eq:skg} by Proposition \ref{prop:inv:ext}, for each $n$ we have the following strict inequality
		\begin{equation}
			\|u(t_n)\|^2_{H^1}+\|v(t_n)\|^2_{H^1}<	\int_{\mathbb{R}^3}\left[u^4(t_n)+v^4(t_n)+2\beta u^2(t_n)v^2(t_n)\right]\mathrm{d}x.
		\end{equation}
		Combining this with \eqref{eq:1403}, we obtain for each $n$
		\begin{equation}\label{eq:1410}
			\left[\|u(t_n)\|^2_{H^1}+\|v(t_n)\|^2_{H^1}\right]\times \left[1-C \left(\|u(t_n)\|^2_{L^2}+\|v(t_n)\|^2_{L^2}\right)^{\frac{1}{2}} \left(\|u(t_n)\|^2_{H^1}+\|v(t_n)\|^2_{H^1}\right)^{\frac{1}{2}}\right]<0
		\end{equation}
		Recalling that $E[u_0,u_1;v_0,v_1]=0$ and the limitation \eqref{eq:1406}, we take $N>0$ so that for all $n \geq N$ there holds
		\begin{equation}
			C \left(\|u(t_n)\|^2_{L^2}+\|v(t_n)\|^2_{L^2}\right)^{\frac{1}{2}} \left(\|u(t_n)\|^2_{H^1}+\|v(t_n)\|^2_{H^1}\right)^{\frac{1}{2}}<1
		\end{equation}
		and hence \begin{equation}
			\left[1-C \left(\|u(t_n)\|^2_{L^2}+\|v(t_n)\|^2_{L^2}\right)^{\frac{1}{2}} \left(\|u(t_n)\|^2_{H^1}+\|v(t_n)\|^2_{H^1}\right)^{\frac{1}{2}}\right]>0.
		\end{equation}
		We then see that \eqref{eq:1410} is impossible, since the product of two non-negative numbers can not be strictly smaller than zero.
	\end{itemize}
	Combining Case 1 and Case 2, we see that $A>0$. This finishes the proof of Claim \ref{claim:6}.
\end{proof}

Let's come back to the proof of \textit{(ii)}. Combing Claim \ref{claim:5} with the strict convexity of $y$, we conclude that $2y(t)-\max\left(0,8E(0)\right)>\delta>0$ if $t\geq t_{\ast\ast}$ for some $t_{\ast\ast}\geq t_{\ast}$.

 Using definitions of $K$ and $E$, together with the conservation of energy, we rewrite \eqref{eq:ene:d2}
\begin{equation} \frac{\mathrm{d}^2y}{\mathrm{d}t^2}(t)=2\int_{\mathbb{R}^3}\left(3|u_t|^2+3|v_t|^2+|\nabla u|^2+|\nabla v|^2+u^2+v^2\right)\mathrm{d}x-8E(0).
\end{equation}
Then for $t\geq t_{\ast\ast}$, we conclude from the discussion in the previous paragraph
\begin{equation}\label{eq:5000}
	\frac{\mathrm{d}^2y}{\mathrm{d}t^2}(t)\geq 2\int_{\mathbb{R}^3}\left(3|u_t|^2+3|v_t|^2\right)\mathrm{d}x+\delta.
\end{equation}
Applying Cauchy-Schwarz inequality on the right hand side of \eqref{eq:ene:d1}, and substituting the resulted inequality into \eqref{eq:5000}, we obtain
\begin{equation}
	\frac{\mathrm{d}^2y}{\mathrm{d}t^2}(t)\geq \frac{3}{2} \frac{\left(\frac{\mathrm{d}}{\mathrm{d}t}y\right)^2}{y(t)}+\delta,\ \ \forall t\geq t_{\ast\ast}.
\end{equation}
At time $t=t_{\ast\ast}$, it follows from the above discussion and \textbf{Claim \ref{claim:5}} that $y(t_{\ast\ast})>0$ and $\frac{\mathrm{d}}{\mathrm{d}t}y(t_{\ast\ast})>0$. At this point, we can apply \textbf{Lemma \ref{lem:levine:1}} with
\[
	y(t)=\psi(t),\ t_b=t_{\ast\ast},\ \gamma=\frac{3}{2}
\]
to conclude that $y(t)$ becomes infinity before some finite time and hence $\left(u(t),v(t)\right)$ blows up in finite time. This contradicts to the initial assumption that the solution exists for all the time. Thus the solution must blow up before some finite time. This completes the proof of \textit{(ii)} in \textbf{Theorem \ref{thm:payne-sattinger:int}}.	
\end{proof}

\section{Remarks on the criteria}
\label{sec:diff}

From Theorems \ref{thm:neg:int},\ref{thm:skg:int} and \ref{thm:payne-sattinger:int}, we see that different conditions on initial data can lead its corresponding solution to blow up in finite time. Then it is natural to ask for the relations between these conditions. This is the first task of this section.

The second task is to give a comment on Theorem \ref{thm:skg:int}, indicating that there does exist initial data satisfying the assumptions there. {The reason for us to do this is that there are too many (three, to be exact) restrictions on the datum, which forces us to make sure that they do not contradict with each other}.

\subsection{Comparing Theorem \ref{thm:neg:int} with Theorem \ref{thm:payne-sattinger:int}}\label{sect:5.1}

In order to finish the first task, we first give a lemma.
\begin{lemma}\label{lem:1to3}
 For $\beta\geq -1$, we have
 \begin{enumerate}[(i)]
 	\item $\{\left((u_0,u_1),(v_0,v_1)\right)\in \mathcal{H}\times\mathcal{H}| E[u_0,u_1;v_0,v_1]<0\}\subset\mathcal{E}$;
 	\item $\{\left((u_0,u_1),(v_0,v_1)\right)\in \mathcal{H}\times\mathcal{H}| E[u_0,u_1;v_0,v_1]=0, \big\langle(u_0,v_0),(u_1,v_1)\big\rangle_{(L^2\times L^2)^2}>0\}\subset  \mathcal{E}.$
 \end{enumerate}
\end{lemma}
\begin{proof}
	Since $E[u_0,u_1;v_0,v_1]\leq0<d$, it suffices to show, if the data $\left((u_0,u_1),(v_0,v_1)\right)\in\mathcal{H}\times\mathcal{H}$ satisfying $E[u_0,u_1;v_0,v_1]<0$ or $E[u_0,u_1;v_0,v_1]=0$ and $\big\langle(u_0,v_0),(u_1,v_1)\big\rangle>0$, then $K[u_0,v_0]<0$.
	
By definitions of $E$ and $K$, we have
\begin{equation}\label{eq:3002}
	E[u_0,u_1;v_0,v_1]=\frac{1}{2}\int_{\mathbb{R}^3} \left[|\nabla u_0|^2+|u_0|^2+|u_1|^2+|\nabla v_0|^2+|v_0|^2+|v_1|^2-\frac{1}{2}\left( u^4_0+v^4_0+2\beta u^2_0v^2_0\right)\right]\mathrm{d}x,
\end{equation}
and
\begin{equation}\label{eq:3003}
	K[u_0,v_0]=\int_{\mathbb{R}^3} \left[|\nabla u_0|^2+|u_0|^2+|\nabla v_0|^2+|v_0|^2\right]\mathrm{d}x-\int_{\mathbb{R}^3}\left[ u^4_0+v^4_0+2\beta u^2_0v^2_0\right]\mathrm{d}x.
\end{equation}
Using the assumption $\beta\geq -1$ and the elementary inequality $a^2+b^2-2ab\geq 0$,
we can get
\begin{equation}\label{eq:3001}
\int_{\mathbb{R}^3}\left[ u^4_0+v^4_0+2\beta u^2_0v^2_0\right]\mathrm{d}x
\geq 0.
\end{equation}

Assume $E[u_0,u_1;v_0,v_1]<0$. We then have
\begin{equation}
	\int_{\mathbb{R}^3}\left[|\nabla u_0|^2+|u_0|^2+|\nabla v_0|^2+|v_0|^2\right]\mathrm{d}x<\frac{1}{2}\int_{\mathbb{R}^3}\left[ u^4_0+v^4_0+2\beta u^2_0v^2_0\right]\mathrm{d}x\leq\int_{\mathbb{R}^3}\left[ u^4_0+v^4_0+2\beta u^2_0v^2_0\right]\mathrm{d}x.
\end{equation}
By connecting the first and the third inequalities, we see that $K[u_0,v_0]<0$, finishing the proof of \textit{(i)}.

To prove \textit{(ii)}, let's assume $E[u_0,u_1;v_0,v_1]=0$ and $\big\langle(u_0,v_0),(u_1,v_1)\big\rangle_{(L^2\times L^2)^2}>0$. In this case, we have the strict inequality in \eqref{eq:3001}. For otherwise, $E[u_0,u_1;v_0,v_1]=0$ implies $$2\int_{\mathbb{R}^3} \left[|\nabla u_0|^2+|u_0|^2+|\nabla v_0|^2+|v_0|^2\right]\mathrm{d}x=\int_{\mathbb{R}^3}\left[ u^4_0+v^4_0+2\beta u^2_0v^2_0\right]\mathrm{d}x=0.$$ But this in turn implies $u_0=v_0=0$ and hence $\big\langle(u_0,v_0),(u_1,v_1)\big\rangle_{(L^2\times L^2)^2}=0$. This contradicts the assumption. Thus we have
\begin{equation}
	\int_{\mathbb{R}^3}\left[ u^4_0+v^4_0+2\beta u^2_0v^2_0\right]\mathrm{d}x
	> 0.
\end{equation}
Therefore
\begin{equation}
		\int_{\mathbb{R}^3}\left[ |\nabla u_0|^2+|u_0|^2+|\nabla v_0|^2+|v_0|^2\right]\mathrm{d}x=\frac{1}{2}\int_{\mathbb{R}^3}\left[ u^4_0+v^4_0+2\beta u^2_0v^2_0\right]\mathrm{d}x<\int_{\mathbb{R}^3}\left[ u^4_0+v^4_0+2\beta u^2_0v^2_0\right]\mathrm{d}x.
\end{equation}
By connecting the first and the third inequalities, we see that $K[u_0,v_0]<0$, finishing the proof of \textit{(ii)}.
\end{proof}

Consequently, for $\beta\geq 0$, by Lemma \ref{lem:1to3}, we see that Theorem \ref{thm:payne-sattinger:int} embodies Theorem \ref{thm:neg:int}, as was commented in \textbf{Remark \ref{rmk:01080}}.


\subsection{Construction of initial data satisfying Theorem \ref{thm:skg:int}}\label{sect:5.2}
This subsection is devoted to constructing initial data with arbitrarily large energy and satisfying all assumptions in Theorem \ref{thm:skg:int}.
We consider the initial data of form
\begin{equation}\label{eq:500}
	(u_0,u_1)=(k_1Q(r),0), (v_0,v_1)=(k_2Q(r),0),
\end{equation}
where $r=\sqrt{x_1^2+x_2^2+x_3^2}, k_1, k_2\in \mathbb{R}$ and $Q(r)\in C^\infty_0(\mathbb{R}_+)$.  Our task then  becomes to seek conditions on $k_1, k_2$ and $Q$ so that the datum $(u_0,u_1)$ and $(v_0,v_1)$ satisfy the following conditions
\begin{enumerate}
\item  the assumptions in Theorem \ref{thm:skg:int} hold;
\item  the energy $E[u_0,u_1;v_0,v_1]$ can be arbitrarily large.
\end{enumerate}

To obtain (1),  we first claim that the first condition and the second condition in Theorem \ref{thm:skg:int} are automatically satisfied, if each of $k_1,k_2$ are not equal to zero and   $Q\not \equiv0$ with compact support.

Next we claim that the third condition in Theorem \ref{thm:skg:int} holds, i.e.
\begin{equation}\label{condi3}
	\frac{1}{2}\int_{\mathbb{R}_+}r^2|Q(r)|^2\mathrm{d}r\geq \int_{\mathbb{R}_+}r^2|Q(r)|^2\mathrm{d}r+\int_{\mathbb{R}_+}r^2|\frac{\mathrm{d} Q}{\mathrm{d}r}(r)|^2\mathrm{d}r-G(\beta,k_1,k_2)\int_{\mathbb{R}_+}r^2|Q(r)|^4\mathrm{d}r>0,
\end{equation}
where $G(\beta,k_1,k_2)=\frac{k_1^2+k_2^2}{2}-\frac{(1-\beta) k_1^2k_2^2}{k_1^2+k_2^2}$.
Indeed, we can show this by using the fact that
\begin{equation*}
	y(0)=\frac{(k_1^2+k^2_2)\omega_3}{2}\int_{\mathbb{R}_+}r^2|Q(r)|^2\mathrm{d}r, ~P(0)=0
\end{equation*}
and $E(0)=\frac{\omega_3}{2}\int_{\mathbb{R}_+}(k^2_1+k_2^2)r^2\left(|Q(r)|^2+|\frac{\mathrm{d} Q}{\mathrm{d}r}(r)|^2\right)\mathrm{d}r-\omega_3\left(\frac{(k_1^2+k^2_2)^2}{4}-\frac{(1-\beta) k_1^2k_2^2}{2}\right)\int_{\mathbb{R}_+}r^2|Q(r)|^4\mathrm{d}r$,
where $\omega_3$ denotes the volume of 2-sphere.

Note that for any fixed $\beta\in \mathbb{R}$, $\mathbb{R}_+$ will fall into the range of $G(\beta,k_1,k_2)$ with respect to variables $k_1$ and $k_2$.
Then for any $Q\not \equiv0$ with compact support, setting $G(\beta,k_1,k_2)=\frac{\frac{1}{2}\int_{\mathbb{R}_+}r^2|Q(r)|^2\mathrm{d}r+\int_{\mathbb{R}_+}r^2|\frac{\mathrm{d} Q}{\mathrm{d}r}(r)
|^2\mathrm{d}r}{\int_{\mathbb{R}_+}r^2|Q(r)|^4\mathrm{d}r}$, we can obtain \eqref{condi3}. This finishes the first part of construction of initial data.

To construct initial data satisfy (2),  let us introduce a smooth cutoff function $\chi(x)\in C_0^\infty (\mathbb{R}_+)$, with the following properties
$\chi(x)\leq 1, \forall x\in \mathbb{R}_+$, {
$\chi(x)=1$  if $|x-3|\leq 1$ and $\chi(x)=0$ if $|x-3|\geq 2$}
We take $Q(r)=\chi\left(\frac{r}{R}\right)$, where $R$ is a constant whose value will be determined below. With this choice of $Q$, substituting \eqref{eq:500} into the defining expression of the energy, we get
\begin{equation*}
	E(0)=\omega_3 \frac{\frac{1}{4}\left(\int_{\mathbb{R}_+}r^2|Q(r)|^2\mathrm{d}r\right)^2+\frac{1}{2}\left(\int_{\mathbb{R}_+}r^2|Q(r)|^2\mathrm{d}r\right)\left(\int_{\mathbb{R}_+}r^2|\frac{\mathrm{d} Q}{\mathrm{d}r}(r)|^2\mathrm{d}r\right)}{\int_{\mathbb{R}_+}r^2|Q(r)|^4\mathrm{d}r}.
\end{equation*}
Then to obtain condition (2),
it suffices to prove
$$for ~any ~\delta>0, ~we ~can ~choose ~R ~such ~that ~E(0)\geq \delta. $$

Indeed, this is achieved by the following computations
\begin{equation*} \int_{\mathbb{R}_+}r^2|Q(r)|^4\mathrm{d}r=\frac{8R^3}{3}+\int_{\mathbb{R}_+\setminus [2R,4R]}r^2\chi^4\left(\frac{r}{R}\right)\mathrm{d}r<\frac{8R^3}{3}+\int_{\mathbb{R}_+\setminus [2R,4R]}r^2\chi^2\left(\frac{r}{R}\right)\mathrm{d}r=\int_{\mathbb{R}_+}r^2|Q(r)|^2\mathrm{d}r.
\end{equation*}
So that $E(0)>\frac{2\omega_3 R^3}{3}$. Thus $E(0)\rightarrow \infty$ as $R\rightarrow \infty$. This ends the construction of initial data.

{Combing this with the fact that initial data with negative initial energy belongs to $\mathcal{E}$,
	we remark that the assumptions of Theorem \ref{thm:skg:int} and Theorem \ref{thm:payne-sattinger:int} are complementary to each other. }


\end{document}